\documentclass[11pt,a4paper]{article}

\usepackage[T1]{fontenc}
\usepackage{amsmath, amsthm}
\usepackage{amsfonts,amssymb}
\usepackage{derivative}
\usepackage{mathtools}
\usepackage[applemac]{inputenc}		
\usepackage{hyperref}
\usepackage{comment}
\usepackage{color}

\usepackage{graphicx}
\usepackage{subcaption}
\mathtoolsset{showonlyrefs}

\hypersetup{colorlinks=true,urlcolor=black, pdftitle="WBMTII"}

\newtheorem{conj}{Conjecture}[section]
\newtheorem{thm}[conj]{Theorem}

\newtheorem{rem}[conj]{Remark}
\newtheorem{lem}[conj]{Lemma}
\newtheorem{prop}[conj]{Proposition}

\newtheorem{defn}[conj]{Definition}
\newtheorem{cor}[conj]{Corollary}

\makeatletter
\newtheorem*{rep@theorem}{\rep@title}
\newcommand{\newreptheorem}[2]{%
\newenvironment{rep#1}[1]{%
 \def\rep@title{#2 \ref{##1}}%
 \begin{rep@theorem}}%
 {\end{rep@theorem}}}
\makeatother
\newreptheorem{theorem}{Theorem}
\newreptheorem{corollary}{Corollary}
\newreptheorem{proposition}{Proposition}

\newcommand{\dlat}{\mathrm{d} }

\newcommand{\spa}{\mathrm{span}}
\newcommand{\conv}{\mathrm{conv}}
\newcommand{\vol}{\mathrm{Vol}}
\newcommand{\supp}{\mathrm{supp}}
\newcommand{\pos}{\mathrm{pos}}
\newcommand{\fcon}{p}

\newcommand{\R}{\mathbb{R}}

 \usepackage[backend=biber,style=numeric-comp,doi=false, bibencoding=utf8,giveninits=true, maxnames=50, isbn=false,url=false]{biblatex}
 \DeclareNameAlias{default}{family-given}
\def\s{\mathbb{S}}

\def\R{{\mathbb R}}

\def\phi{\varphi}

\def\bee{\begin{eqnarray*}}
\def\ene{\end{eqnarray*}}

\newcommand\nnfootnote[1]{%
  \begin{NoHyper}
  \renewcommand\thefootnote{}\footnote{#1}%
  \addtocounter{footnote}{-1}%
  \end{NoHyper}
}

\begin{document}

\title{Gr\"unbaum's inequality for Gaussian \\and convex probability measures}

\author{Matthieu Fradelizi, Dylan Langharst,
Jiaqian Liu,
\\
Francisco Mar\'in Sola, and Shengyu Tang}

\date{\today}
\maketitle
\begin{abstract}
A celebrated result in convex geometry is Gr\"unbaum's inequality, which quantifies how much volume of a convex body can be cut off by a hyperplane passing through its barycenter. In this work, we establish a series of sharp Gr\"unbaum-type inequalities - with equality characterizations - for probability measures under certain concavity assumptions. As an application, we apply the renowned Ehrhard inequality and deduce an ``Ehrhard-Gr\"unbaum'' inequality for the Gaussian measure on $\mathbb{R}^n$, which improves upon the bound derived from its log-concavity.

For $s$-concave Radon measures, our framework provides a simpler proof of known results and, more importantly, yields the previously missing equality characterization. This is achieved by gaining new insight into the equality case of their Brunn-Minkowski-type inequality. Moreover, we show that these ``$s$-Gr\"unbaum'' inequalities can hold only when $s > -1$. However, for convex measures on the real line, we prove Gr\"unbaum-type inequalities involving their cumulative distribution function.

\end{abstract} 

\tableofcontents

\nnfootnote{Keywords: Gr\"unbaum's inequality, Gaussian measure, Ehrhard's inequality, weighted Brunn-Minkowski inequality, Gaussian barycenter.

}

\section{Introduction}
Gr\"unbaum's inequality \cite{GB60} states that for a convex body $K$ (compact, convex set with non-empty interior) in $\R^n$ (the $n$-dimensional Euclidean space), one has
\begin{equation}
\frac{\vol_{n}(K \cap H^{-})}{\vol_n(K)} \geq \left(\frac{n}{n+1}\right)^{n},
\label{eq:grun_classic}
\end{equation}
where $\vol_k(\cdot)$ is the $k$-dimensional Lebesgue measure, $H$ is any $(n-1)$-dimensional hyperplane through the barycenter of $K$ and $H^{-}$ is one of the half-spaces in $\R^n$ whose boundary is $H$. Recall that the barycenter of a convex body $K$ with respect to the Lebesgue measure is precisely  the vector $$g(K)=\int_{K}x\frac{dx}{\vol_n(K)}.$$
Equality holds in \eqref{eq:grun_classic} if and only if $K$ is a compact truncated cone with basis parallel to $H$ and apex in $H^-$, i.e. $K=\conv (\{b\}, K\cap (H +\ell)) $ for some $b,\ell\in\R^n$, with $b\in H^-$. Gr\"unbaum's inequality \eqref{eq:grun_classic} is essentially a consequence of Brunn's concavity principle, which is equivalent to the Brunn-Minkowski inequality when dealing with convex bodies. We refer the interested reader to the excellent survey by R. Gardner \cite{G20} for the history of the Brunn-Minkowski inequality.

Gr\"unbaum's result has attracted a lot of attention during the last years; for instance, it was extended to the case of sections \cite{FMY17,MSZ18} and projections \cite{SNZ17} of compact convex sets, and generalized to the analytic setting of \emph{log-concave} functions \cite{MNRY18} and $p$-concave functions \cite{MSZ18,AMY25} (see also \cite{MY21,MF24}), for $p>0$. Other Gr\"unbaum-type inequalities can be found in \cite{SY19}.

In this work, we aim to establish versions of \eqref{eq:grun_classic} for certain Borel measures on $\R^n$. For a measurable set $K$ satisfying $0<\mu(K)<\infty$, we set the barycenter of $K$ with respect to $\mu$ as
\[
g_{\mu}(K) = \frac{1}{\mu(K)}\int_{K}xd\mu(x).
\]
Moreover, $g_{\mu}(K)$ exists if $\int_K|\langle x,u \rangle|d\mu(x) < +\infty$ for all $u\in \s^{n-1}$, the unit (Euclidean) sphere. Here, we have denoted the usual inner-product on $\R^n$ by $\langle \cdot,\cdot\rangle$. Also, if $H$ is a hyperplane with outer-unit normal $u\in\s^{n-1}$ at a distance $a\in\R$ from the origin, then, the lower-and-upper half-spaces in $\R^n$ with boundary $H$ are given by
\[
H^-=\{x\in\R^n:\langle x,u \rangle \leq a\}\; \text{ and} \quad H^+=\{x\in\R^n:\langle x,u \rangle \geq a\}.
\]
\begin{thm}
\label{t:best_2}
    Suppose $\mu$ is a Borel measure on $\R^n$ with density. Let $K\subset \R^n$ be a measurable set so that $0<\mu(K)<+\infty$, and suppose $g_\mu(K)$ exists. Let $H$ a hyperplane through $g_{\mu}(K)$ with outer-unit normal $u$. 
    
    Let $F:(0,\mu(K))\to \R$ be an invertible, continuous function. 
    Suppose that $F$ is increasing (resp. decreasing).
    If the function $r\mapsto (F\circ \mu) \left(\left\{x\in K:\left\langle x,u\right\rangle \leq r\right\}\right)$
    is concave (resp. convex) on its support, then one has the inequality
    \begin{equation}
    \label{eq:lower_half_inequality}
    \mu(K\cap H^-)\geq F^{-1}\left(\int_{0}^{\mu(K)}F(r) \frac{dr}{\mu(K)}\right),
    \end{equation}
with equality if and only if the associated concave or convex function is affine.
\end{thm}

We have three main applications of Theorem~\ref{t:best_2}. Firstly, we prove the following within the class of one dimensional \emph{convex measures}, (i.e., those ($-\infty$)-concave measures) by changing our perspective to more probabilistic notions. We need to introduce some terminology that will be of central importance. Given a probability measure $\mu$ on $\R$ with density $\psi$, we define its cumulative distribution function (CDF) as
\begin{equation}
\Phi_\mu(t)=\int_{-\infty}^t \psi(r)dr.
\label{eq:CDF}
\end{equation}
We define the numbers
\[
\alpha_\mu = \inf\{t\in \R:\Phi_{\mu}(t)>0\} \quad \text{and} \quad \beta_\mu = \sup\{t\in\R:\Phi_{\mu}(t)<1\},
\]
i.e. the (possibly unbounded) interval $(\alpha_\mu,\beta_\mu)$ is the smallest interval containing the support of $\mu$, and, if the support of $\mu$ is connected, then they are equal.

\begin{thm}
\label{t:one_d}
    Let $\mu$ be a convex probability measure on $\R$ with density whose barycenter $g_\mu$ exists. Consider an interval $(a,b)\subset\R$ such that $\mu((a,b))>0$ and set
    $$g = \frac{\int_a^brd\mu(r)}{\int_a^bd\mu(r)} \in (a,b) \quad \text{and} \quad t=\mu((a,b)).$$
    Then,
    \begin{equation}
        \mu((a,g])\geq \Phi_{\mu}\left(\int_0^t\Phi_{\mu}^{-1}(r)\frac{dr}{t}\right).
        \label{eq:one_d}
    \end{equation}

There is equality if and only if $a\leq \alpha_\mu$.
\end{thm}
\noindent In Proposition~\ref{p:better_right_hand} below, we establish that $\int_0^t\Phi_{\mu}^{-1}(r) dr$ is finite if and only if $g_\mu$ exists. One can find other works on the interaction between convex geometry and probability theory in \textit{e.g.} \cite{BS07,FLM20,MMX17,BM12,LYZ04_3,LYZ05_2,LYZ07,LLYZ13,LLYZ12}. 

Next, we consider the standard Gaussian measure on $\R^n$: $\gamma=\gamma_1$ and, for $n\geq 1$,
\begin{equation}
    d\gamma_n(x)=(2\pi)^{-\frac{n}{2}}e^{-\frac{|x|^2}{2}}dx,
    \label{eq:Gaussian}
\end{equation}
where $|\cdot|$ is the usual Euclidean norm. Let $\Phi$ be the cumulative distribution function of the one-dimensional normal distribution, i.e. 
\begin{equation}
    \Phi(t)=\gamma((-\infty,t])=\frac{1}{\sqrt{2\pi}}\int_{-\infty}^t e^{-\frac{s^2}{2}}ds. 
\end{equation}
Of course, we have that $\Phi=\Phi_{\gamma}$. Then, the extension of Ehrhard's inequality by C. Borell \cite{Bor03} states that if $A$ and $B$ are two Borel sets in $\R^n$ and $\lambda\in [0,1]$, then
\begin{equation}
\label{eq:ehrhard}
    \Phi^{-1}\left(\gamma_n\left((1-\lambda)A+\lambda B\right)\right) \geq (1-\lambda)\Phi^{-1}\left(\gamma_n(A)\right) + \lambda \Phi^{-1}\left(\gamma_n(B)\right).
\end{equation}
This extends on the special case of two closed, convex sets \cite{EHR1} due to A. Ehrhard. We note that equality characterization is open in general, but in the setting of convex sets with positive Gaussian measure there is equality if and only if \cite{EHR2} the two sets coincides, or one of the sets is $\R^n$, or both sets are half-spaces, one contained in the other. It is well-known that Ehrhard's inequality implies Gaussian isoperimetry:
for all Borel sets $A\subset\R^n$ and $h>0$, it holds
\begin{equation}
\label{eq:Gaussian_iso_inequality}
    \frac{\gamma_n( A+hB_2^n)-\gamma_n(A)}{h} \geq \frac{\Phi\left(\Phi^{-1}(\gamma_n(A))+h\right)-\gamma_n(A)}{h},
\end{equation}
with equality if and only if $A$ is a half-space. Here, $B_2^n$ is the unit Euclidean ball. Sending $h\to 0$ in \eqref{eq:Gaussian_iso_inequality}, the left-hand side approaches the Gaussian Minkowski content of $A$, and right-hand side approaches $I_\gamma(\gamma_n(A))$, where $I_\gamma$ is the isoperimetric function of the Gaussian measure given by
\begin{equation}
    I_{\gamma}(t):=\phi\left(\Phi^{-1}(t)\right), \; \text{where} \quad \phi(t)=\Phi^\prime(t) = (2\pi)^{-\frac{1}{2}}e^{-\frac{t^2}{2}}.
    \label{eq:gaussian_iso}
\end{equation}

Theorem \ref{t:best_2} together with \eqref{eq:ehrhard} are used to show the following, where we call $g_{\gamma_n}(K)$  the Gaussian barycenter of $K$.
\begin{thm}
\label{t:Gaussian_Grun}
    Let $K$ be a convex set in $\R^n$ with non-empty interior. Suppose $H$ is a hyperplane through the Gaussian barycenter of $K$. Then, it holds
    \begin{equation}
       \gamma_{n}(K\cap H^{-}) \geq \Phi\left(-\frac{I_{\gamma}(t)}{t}\right); \quad \text{with}\; t=\gamma_n(K).
        \label{eq:Grunbaum_Gaussian}
    \end{equation}
    There is equality if and only if $K$ is half-space parallel to $H$ containing $H^{-}$.
\end{thm}

\noindent The quantity $-\frac{I_{\gamma}(t)}{t}$ has a different representation via a change of variables: for $t\in [0,1]$, it holds
    \begin{equation}
    \label{eq:different_Ehr_iso}
        -\frac{I_{\gamma}(t)}{t} =\int_0^t\Phi^{-1}(r)\frac{dr}{t}.
    \end{equation}

Theorem~\ref{t:Gaussian_Grun} continues the investigation of \cite{HXZ21,LRZ22,FLMZ23_1,FLMZ23_2,LP25}, which extends different notions of the classical Brunn-Minkowski theory to the Gaussian measure via Ehrhard's inequality \eqref{eq:ehrhard}. Other investigations concerning the Gaussian measure, are plentiful; see \textit{e.g.} \cite{CEFM04,GZ10,KL21,EM21,CER23}. Our final result concerns inequalities for $s$-concave probability measures. Recall that, given a probability measure $\mu$ on $\R^n$, the barycenter of $\mu$ is the vector
\[
g_\mu=\int_{\R^n}xd\mu(x).
\] 
Moreover, the barycenter exists if $\int_{\R^n}|\langle x,u\rangle|d\mu(x)<+\infty$ for every $u\in \s^{n-1}$. We say a Borel measure has density if it has a locally integrable Radon-Nikodym derivative with respect to the Hausdorff measure of the affine subspace containing the support of $\mu$.

Gr\"unbaum-type inequalities for $s$-concave measures (See Definition~\ref{def:s_concave} in Section~\ref{s:s_concave_def} below) were first established for $s>0$ by S. Myroshnychenko, M. Stephen, and N. Zhang in \cite[Corollary 7]{MSZ18} and for $s>-1$, in the more general framework of RCD spaces, by V. Brunel, S. Ohta and J. Serres \cite{BOS24}.  Our contributions are twofold; on the one hand  we shall give a very short proof of Gr\"unbaum's inequality for $s$-concave measures, with $s> -1$, as a consequence of Jensen's inequality which heavily simplifies the previous techniques. On the other hand, we provide a complete equality characterization, which was the missing piece in previous works. This required establishing new facts concerning equality in the Borell-Brascamp-Lieb inequality.

\begin{thm}
\label{t:s_concave}
Let $s>-1$ and $\mu$ be a $s$-concave, Radon, probability measure on $\R^n$. Fix $u\in \s^{n-1}$ and define the hyperplane $H=g_{\mu}+u^\bot$. Then, for $s\not = 0$:
    \begin{equation}
    \label{eq:Grun_s_neq_0}
   \mu( H^-)\geq \left(\frac{1}{1+s}\right)^{\frac{1}{s}}.
    \end{equation}

    For $s=0$, $\mu$ is log-concave and we obtain
    \begin{equation}\label{Gruabaum inequality for log-concave}
       \mu( H^-)\geq \frac{1}{e}.
    \end{equation}   
\noindent Moreover, let $\psi$ be the density of $\mu$, which exists by Borell's classification (Corollary~\ref{cor:BBL} below) and is concentrated on an affine subspace $E$ of $\R^n$ whose dimension is  $d\in\{1,\dots,n\}$. Then, there is equality if and only if there exists a non-negative, $p$-concave function $w$ on $u^\bot\cap E$, where $p=\frac{s}{1-ds}$, a vector $v\in E$ satisfying $\langle u,v\rangle=1$, and numbers $a>0$ and  $r_1<\infty$ such that, for almost all $z\in u^\bot \cap E$ and $r\in \R$, one has:
\begin{enumerate}
    \item If $s>0$, then, for $r>r_1-\frac{1}{a}:=r_0$,
        \begin{equation}
\label{eq:psi_forumla_pos}
\psi(z+rv)=\left(1+a(r-r_1)\right)^\frac{1}{\fcon}w\left(\frac{z}{1+a(r-r_1)}\right)\chi_{(-\infty,r_1]}(r),
\end{equation}
for $r< r_0$, we have $\psi(z+rv) =0$, and finally, for $r=r_0$, either $\psi$ is constant on its support and $s=\frac{1}{d}$ (in which case, $w$ is the characteristic function of a $(d-1)$-dimensional convex set) or $s\in (0,\frac{1}{d})$ and $\psi\left(z+r_0v\right)=0.$
\item If $s=0$, then,
\begin{equation}
\label{eq:psi_affine_s_0_intro}
\psi(z+ r v)=e^{a(r-r_1)}w(z)\chi_{(-\infty,r_1]}(r).
\end{equation}

\item If $s\in (-1,0)$, then
\begin{equation}
\label{eq:psi_affine_neg_intro}
    \psi(z+rv)=\left(1+a(r-r_1)\right)_+^\frac{1}{\fcon}w\left(\frac{\mathcal{R}-r_1}{\mathcal{R}-r}z\right)\chi_{(-\infty,r_1]}(r)
\end{equation}
for some $\mathcal{R}>r_1$.

\end{enumerate}
\end{thm}

We remark that for the proof of the inequality in Theorem~\ref{t:s_concave}, we do not need $\mu$ to be $s$-concave with respect to Borel sets, but only with respect to convex sets. In fact, it is only needed that $r\mapsto\mu(\{x; \langle x,u\rangle\le r\})$ is $s$-concave. Also, the Gaussian measure $\gamma_n$ is $\log$-concave, i.e. we can apply \eqref{Gruabaum inequality for log-concave} with $\mu=\gamma_n$ and obtain
\begin{equation}\label{e:Gaussian_Gr}
    \gamma_n(H^-)\geq \frac{1}{e}.
\end{equation}
However, we emphasize that inequality \eqref{eq:Grunbaum_Gaussian} is stronger than \eqref{e:Gaussian_Gr}. This amounts to verifying \begin{equation}
    \label{eq:inequality}
    \frac{t}{e}\leq \Phi\left(- \frac{I_\gamma(t)}{t}\right)=\Phi\left(\int_0^t\Phi^{-1}(r)\frac{dr}{t}\right)
    \end{equation}
    on $[0,1]$. Recall that $\log \Phi$ is concave, and thus, by using Jensen's inequality, we have
    \[
    \log(t)-1 \leq (\log\circ \Phi)\left(\int_0^t\Phi^{-1}(r)\frac{dr}{t}\right).
    \]
    Then, \eqref{eq:inequality} follows by exponentiating both sides.

Concerning the equality cases for Theorem~\ref{t:s_concave}, we take a moment to explain how to read off what the support of $\mu$ is from the given formulas when $d=n$. In particular, we verify that the classical Gr\"unbaum inequality \eqref{eq:grun_classic} is recovered from Theorem~\ref{t:s_concave} by taking $\mu$ to be the uniform measure of a convex body $K$, i.e., $d\mu(x)=\frac{1}{\vol_n(K)}\chi_K(x)dx$, and $s=\frac{1}{n}$. Here, $\chi_K$ is the usual characteristic function of $K$; it equals $1$ on $K$ and $0$ off $K$.

Recall the convex hull of a set $A$ is precisely
\[
\conv(A)=\left\{(1-\lambda)x+\lambda y:\lambda \in [0,1],x,y\in A\right\},
\]
and the convex hull of two sets $A_1$ and $A_2$ is $\conv(A_1,A_2):=\conv(A_1\cup A_2)$.
Similarly, the positive hull of a set is 
\[
\pos(A) = \{\lambda x:\lambda \geq 0, x\in A\}.
\]
The support function of a convex set $K\subset\R^n$ is precisely $h_K(u)=\sup_{y\in K}\langle y,u\rangle$ for $u\in\s^{n-1}$. If $K$ is an unbounded convex set, then $h_K$ may take on values of $\pm \infty$. 
 Henceforth, set $K=\supp(\mu)$. Supposing there is equality in Theorem~\ref{t:s_concave}, $h_K(u):=r_1$ is finite. We set $$L:=\supp(\mu)\cap \{x\in\R^n:\langle x,u\rangle =r_1\},$$ which is an $(n-1)$-dimensional convex body in a hyperplane parallel to $H$.

\begin{figure}[htb]
  \centering
  \begin{subfigure}[b]{0.48\textwidth}
    \centering
    \includegraphics[width=\textwidth]{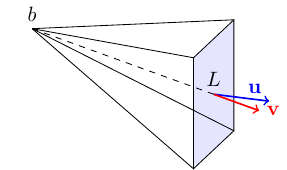}
    \caption{Rectangular face}
    \label{fig:s_gt_0_rect}
  \end{subfigure}
  \hfill
  \begin{subfigure}[b]{0.48\textwidth}
    \centering
    \includegraphics[width=\textwidth]{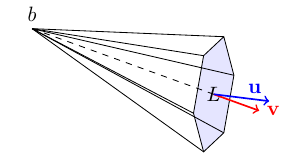}
    \caption{Hexagonal face}
    \label{fig:s_gt_0_hex}
  \end{subfigure}
  \caption{The case of equality when $s > 0$.}
  \label{fig:s_gt_0}
\end{figure}

In the case $s>0$, we have that $K$ is a compact cone in the direction $v$, whose face is orthogonal to $u$. Set $r_0:=r_1-\frac{1}{a}$ and $b:=r_0v$. Then, we claim $K=\conv(L,\{b\})$, which, in turn, yields that $r_0= -h_K(-u)$. Indeed, first notice that $$(1+a(r-r_1))_+ \quad \text{ rewrites as } \quad \left(\frac{r-r_0}{r_1-r_0}\right)_+.$$ Therefore, for $\psi(z+rv)>0$, one must have $\langle b,u\rangle=r_0 \leq r \leq r_1$. Next, $$w\left(\frac{z}{1+a(r-r_1)}\right) \quad \text{re-writes as }\quad w\left(\left(\frac{r_1-r_0}{r-r_0}\right)z\right),$$ and $\supp(w)=L-r_1 v$. Consequently, 
\[
\left(\frac{r_1-r_0}{r-r_0}\right)z \in L -r_1v \longleftrightarrow z+rv\in \left(1 - \frac{r_1-r}{r_1-r_0}\right) L + \frac{r_1-r}{r_1-r_0}b.
\]
Therefore,
\begin{align*}
    K=\cup_{r_0 \leq r \leq r_1}\left\{\left(1 - \frac{r_1-r}{r_1-r_0}\right) L + \frac{r_1-r}{r_1-r_0}b\right\}=\conv(L,\{b\}).
\end{align*}

For $s=0$, $K$ is a cylinder in the direction $v$ truncated in the direction $u$. That is, $K=L+(-\pos(v))$. Indeed, $\psi(z+rv)>0$ if and only if $r\leq r_1$ and $z+r_1v \in L$. Finally, for such $z$ and $r$ notice that $z+rv = (z+r_1v)-(r_1-r)v \in L+(-\pos(v))$.
\begin{figure}[htb]
  \centering
  \begin{subfigure}[b]{0.48\textwidth}
    \centering
    \includegraphics[width=\textwidth]{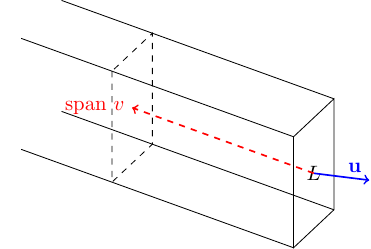}
    \caption{Rectangular cross section}
    \label{fig:s_eq_0_rect}
  \end{subfigure}
  \hfill
  \begin{subfigure}[b]{0.48\textwidth}
    \centering
    \includegraphics[width=\textwidth]{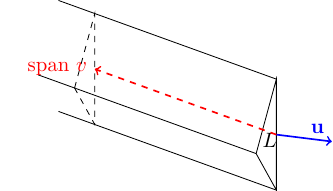}
    \caption{Triangular cross section}
    \label{fig:s_eq_0_tri}
  \end{subfigure}
  \caption{The case of equality when $s = 0$. Sections are translates of the face along $-v$.}
  \label{fig:s_eq_0}
\end{figure}

For $s<0$, $K$ is a truncated cone in the direction $v$. Set $b:=\mathcal{R}v$. Observe that
\begin{align*}
\frac{\mathcal{R}-r_1}{\mathcal{R}-r}z \in L-r_1v \longleftrightarrow z+rv \in L-\pos\left(\frac{b-L}{\mathcal{R}-r_1}\right).
\end{align*}
Taking union over all $r\leq r_1$, we have $K=L-\pos\left(\frac{b-L}{\mathcal{R}-r_1}\right)$.
\begin{figure}[htbp]
  \centering
  \begin{subfigure}[b]{0.48\textwidth}
    \centering
    \includegraphics[width=\textwidth]{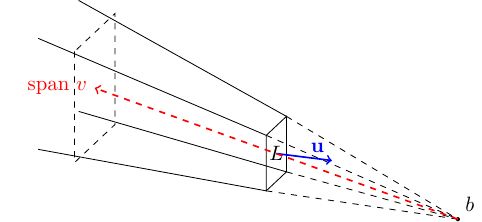}
    \caption{Rectangular cross section}
    \label{fig:s_lt_0_rect}
  \end{subfigure}
  \hfill
  \begin{subfigure}[b]{0.48\textwidth}
    \centering
    \includegraphics[width=\textwidth]{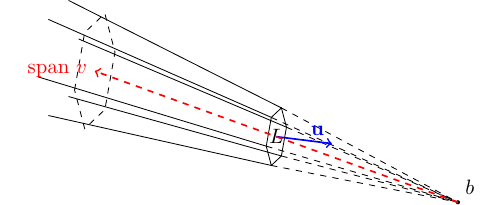}
    \caption{Hexagonal cross section}
    \label{fig:s_lt_0_tri}
  \end{subfigure}
  \caption{The case of equality when $s < 0$. Sections are homothetic expansions of the face.}
  \label{fig:s_lt_0}
\end{figure}

\noindent Note, by defining $d=\mathcal{R}-r_1$ and $M=\frac{b-L}{\mathcal{R}-r_1}$, we can equivalently say there is equality if and only if $K=b-\left(d\cdot M+\pos(M)\right)$ for some $d>0$ and $(n-1)$-dimensional compact, convex set $M$ parallel to $u^\perp$. However, this obscures the geometric meaning of $M$ and $d$.

The reader may have noticed that, as $s\to -1$, $\left(\frac{1}{1+s}\right)^\frac{1}{s}\to 0$, and so the inequality \eqref{eq:Grun_s_neq_0} becomes trivial. In the following proposition, whose proof is postponed until Section~\ref{sec:very_neg}, we show that an inequality as \eqref{eq:Grun_s_neq_0} does not hold in general for $s$-concave measures when $s \leq -1$. We shall show only the $n=1$ case; these can be embedded into $\R^n$ by revolving the density. 
\begin{prop}
    \label{p:no_bound}
    Fix $s \leq -1$. Then, there exists a sequence $\{\mu_k\}_{k=1}^\infty$ of $s$-concave probability measures on $\R$ such that
    \[
    \lim_{k\to +\infty }\mu_k((-\infty,g_{\mu_k}])=0.
    \]
\end{prop}
This is in sharp contrast to Theorem~\ref{t:one_d}, which showed such measures \textit{do} satisfy a Gr\"unbaum-type inequality, if we use the inverse of their cumulative distribution function instead of the function $F(t)=t^s$.

As a final remark, we elaborate on how Theorem~\ref{t:s_concave} compares and completes known results. Firstly, we recall that the center of mass of a function is defined as
$$
g_f = \frac{\int_{\R^n} xf(x) dx}{\int_{\R^n} f(x)dx}.
$$
As aforementioned, S. Myroshnychenko, M. Stephen, and N. Zhang established in \cite[Corollary 7]{MSZ18} the following  Gr\"unbaum-type inequality (see also \cite{MF23Th} for a different proof): suppose that $f$ is $\fcon$-concave, $\fcon>0$, with $g_f=o$ and $H$ is any hyperplane through the origin, then
\begin{equation}\label{e:GrFunc_general_n_question}
\int_{H^{-}}f(x)\,\dlat x\geq C(n,\fcon)\int_{\R^n}f(x)\,\dlat x, \quad C(n,\fcon)=\left(\frac{n\fcon+1}{(n+1)\fcon+1}\right)^{\frac{n\fcon+1}{\fcon}}, C(n,0)=\frac{1}{e}.
\end{equation}
The log-concave case (i.e., when $\fcon=0$), was due to Lov\'asz and Vempala \cite[Lemma~5.4]{LV07} -see also \cite{MNRY18}.  More recently, Gr\"unbaum's inequality for $\mathrm{RCD}(0,N)$ spaces was established in \cite{BOS24}. When restricted to the real Euclidean space, their result extends \eqref{e:GrFunc_general_n_question} to $\fcon > -\frac{1}{2}$. Observe that Theorem~\ref{t:s_concave} implies \eqref{e:GrFunc_general_n_question} by taking $\mu$ to be the measure with density $\frac{f(x)dx}{\int_{\R^n}f(x)dx}$ and using Corollary~\ref{cor:BBL} below to identify $s=\frac{\fcon}{1+n\fcon}$. Note also that the condition $\fcon > -\frac{1}{2}$ is precisely $s>-1$. Therefore, our results formally recover \eqref{e:GrFunc_general_n_question} for all known $\fcon$. 

Finally, we list here the $n=1$ case, which deserves its own discussion:  set $(\alpha,\beta)=\supp f$, in which case the inequality \eqref{e:GrFunc_general_n_question} becomes
\begin{equation}
    \label{eq:Grun_functional}
    \int_{\alpha}^{g_f}f(r) dr\geq \left(\frac{\fcon+1}{2\fcon+1}\right)^\frac{\fcon+1}{\fcon}\int_\alpha^\beta f(r) dr,
    \end{equation}
    with equality if and only if $f$ is $\fcon$-affine on $(\alpha,\beta)$. In the case $\fcon=0$, i.e. $f$ is a log-concave function, this becomes
    \begin{equation}
        \label{e:GrFunc_n=1_logconcave}
        \int_{\alpha}^{g_f}f(r) dr\geq \frac{1}{e}\int_\alpha^\beta f(r) dr,
    \end{equation}
with equality if and only if $f$ is log-affine on $(\alpha,\beta)$, These inequalities had been recently re-established in \cite{MY21}. Formally, it also implies the inequality \eqref{eq:Grun_s_neq_0} (for all $n$) by picking $$f(t)=\int_{\{z\in u^\perp:z+tu\in \supp(\mu)\}}\psi(z+tu)dz.$$ However, showing that this particular function $f$ being $p$-affine implies the listed equality characterization of Theorem~\ref{t:s_concave} is precisely the heart of the matter, and it was not a straightforward task.  

This work is organized as follows. In Section~\ref{sec:main}, we prove first one-dimensional results. In particular, we establish Theorem~\ref{t:one_d}.  In Section~\ref{sec:actual_main}, we establish some foundational results for larger dimensions, establishing Theorem~\ref{t:best_2} along the way. Section~\ref{sec:s_concave} is dedicated to the proof of Theorem~\ref{t:s_concave}. In Section~\ref{sec:gaussian}, we consider the Gaussian measure and prove Theorem~\ref{t:Gaussian_Grun}; we then show that, if the transport map $T$ from $\gamma$ onto a measure $\mu$ on the real-line is convex, then \eqref{eq:Grunbaum_Gaussian}, but with $\mu$ in place of $\gamma$, holds. 

\section{One dimensional results}
\label{sec:main}
As pointed out in the introduction,  Gr\"unbaum-type inequalities can be in general reduced to their $n=1$ case via marginals. We start this section by proving the one dimensional case of Theorem \ref{t:best}. 
\begin{lem}
\label{l:best}
    Suppose $\mu$ is a probability measure on $\R$ with density whose barycenter $g_\mu$ exists. Let $F:(0,1)\to \R$ be an invertible, continuous function and suppose that $F$ is increasing (resp. decreasing). If $F\circ \Phi_\mu$ is concave (resp. convex) on $(\alpha_\mu,\beta_\mu)$, then one has the inequality
    \begin{equation}
    \label{eq:lower_half_inequality_one}
    \mu((\alpha_\mu,g_\mu])\geq F^{-1}\left(\int_{0}^{1}F(r) dr\right).
    \end{equation}
   There is equality if and only if $F\circ \Phi_\mu$ is affine.
\end{lem}
\begin{proof}
    Let $\psi$ be the density of $\mu$. Notice that
    \begin{equation}\label{express of Halfspace intergral}
        \mu((\alpha_\mu,g_\mu])=\Phi_\mu(g_\mu).
    \end{equation}
Let $\bar F$ be the following primitive of $F$ 
\[
\bar F(r)=\int_0^r F(s)ds.
\]
    Then, Jensen's inequality and integration by parts yields, if $F$ is increasing,
    \begin{align*}
        &(F\circ \Phi_\mu)(g_\mu)=(F\circ \Phi_\mu)\left(\int_{\R}r d\mu(r)\right)\geq\int_{\R}(F\circ \Phi_\mu)(r)d\mu(r)
        \\
        &=\int_{\alpha_\mu}^{\beta_\mu}(F\circ \Phi_\mu)(r)\psi(r)dr=[(\bar F\circ \Phi_\mu)(r)]|_{\alpha_\mu}^{\beta_\mu}=\bar F(1).
    \end{align*}
    The opposite inequality holds if $F$ is decreasing.
    Since $F$ is strictly monotone and continuous, $F$ is invertible. This implies inequality \eqref{eq:lower_half_inequality_one} in either case of monotonicity. 
\end{proof}

  We now prove Theorem~\ref{t:one_d}. First, we need the following lemma. In Definition~\ref{def:s_concave} below, we formally define convex measures; in the current setting of a probability measure $\mu$ on $\R$ with density $\psi$, the measure $\mu$ is convex if $\psi$ is $(-1)$-concave, i.e. $\left(\frac{1}{\psi}\right)$ is a convex function.
\begin{lem}
\label{l:half_space_concave}
    Let $\mu$ be a probability measure on $\R$ with density that is differentiable almost everywhere. Then, the following are equivalent:
    \begin{enumerate}
        \item $\mu$ is convex.
        \item For every $a \in \R$, the function \begin{equation}
        f(r)=\Phi_\mu^{-1}(\Phi_\mu(r)-\Phi_\mu(a))=\Phi_\mu^{-1}(\mu(a,r])
        \label{eq:concave}
    \end{equation} is concave on its support. 
    \end{enumerate}
\end{lem}
\begin{proof}
We start by writing claim 2.) in an equivalent way. Let $\psi$ be the density of $\mu$. For ease, we define the function
    \begin{equation}
    \label{eq:I_mu}
        I_{\mu}(r)=\psi\left(\Phi_\mu^{-1}(r)\right),
    \end{equation}
    which is of course inspired by the definition of $I_{\gamma}$. We compute the derivative of $I_{\mu}$ and get, for almost all $r$,
    \begin{equation}
    \label{eq:I_mu_deriv}
        I_{\mu}^\prime(r) = \frac{\psi^{\prime}\left(\Phi_\mu^{-1}(r)\right)}{I_{\mu}(r)}.
    \end{equation}
    Differentiating $f$, we obtain
    \begin{equation}
        f^\prime(r) = \frac{\psi(r)}{I_{\mu}(\Phi_\mu(r)-\Phi_\mu(a))}.
    \end{equation}
    Differentiating again with the aid of \eqref{eq:I_mu_deriv} yields, for almost all $r$,
    \begin{equation}
    \label{eq:second_f_deriv}
        f^{\prime\prime}(r) = \frac{\psi^{^\prime}(r)I_{\mu}(\Phi_\mu(r)-\Phi_\mu(a))^2-\psi(r)^2\psi^{\prime}\left(\Phi_\mu^{-1}(\Phi_\mu(r)-\Phi_\mu(a))\right)}{I_{\mu}(\Phi_\mu(r)-\Phi_\mu(a))^3}.
    \end{equation}
    The second derivative of $f$ given by \eqref{eq:second_f_deriv} is non-positive, i.e. claim 2.) holds, if and only if
    \begin{equation}
    \label{eq:concave_iff}
        \frac{\psi^{^\prime}(r)}{\psi(r)^2}\leq \frac{\psi^{\prime}\left(\Phi_\mu^{-1}(\Phi_\mu(r)-\Phi_\mu(a))\right)}{\psi^2\left(\Phi_\mu^{-1}(\Phi_\mu(r)-\Phi_\mu(a))\right)},
    \end{equation}
    where we inserted the definition of $I_{\mu}$ from \eqref{eq:I_mu}. 
    
    Suppose now that $\mu$ is convex. Then, $\left(\frac{1}{\psi}\right)$ is convex, which means the function 
    \[
    r\mapsto \frac{\psi^\prime(r)}{\psi(r)^2}
    \]
    is decreasing on $(\alpha_\mu,\beta_\mu)$. But, observe that, since $\Phi_\mu$ is nonnegative, we have $\Phi_\mu(r)-\Phi_\mu(a)\leq \Phi_\mu(r),$ and thus we deduce $\Phi_\mu^{-1}\left(\Phi_\mu(r)-\Phi_\mu(a)\right)\leq r$. Therefore, \eqref{eq:concave_iff} holds.
    
    Conversely, suppose we know \eqref{eq:concave_iff}. Let $t \leq r$. By letting $a= \Phi_{\mu}^{-1}\left(\Phi_{\mu}(r)-\Phi_{\mu}(t)\right)$, we obtain
    \[
    \frac{\psi^{^\prime}(r)}{\psi(r)^2}\leq \frac{\psi^{^\prime}(t)}{\psi(t)^2}.
    \]
     By definition, this means $\psi$ is $(-1)$-concave, i.e. $\mu$ is convex. We have thus shown the equivalence between 1.) and 2.).
\end{proof}
\begin{rem}
The formula \eqref{eq:concave} yields a Brunn-Minkowski-type inequality for the CDF a convex, probability measure on $\R$. Indeed, \eqref{eq:concave} is equivalent to the inequality: for every $a \leq \beta_\mu,$ $r,t\in (\max\{a,\alpha_\mu\},\beta_\mu)$ and $\lambda\in[0,1]$ one has 
    \begin{equation}
    \mu\left((a, (1-\lambda)r+\lambda t]\right)
        \geq \Phi_\mu\left((1-\lambda)\cdot \Phi_\mu^{-1} \left(\mu((a,r]\right))+\lambda \cdot \Phi_\mu^{-1}\left(\mu((a,t])\right)\right).
        \label{eq:half_space_convex}
    \end{equation}

     We can also handle the equality characterization. For fixed $a \leq \beta_\mu$, there is equality in \eqref{eq:half_space_convex} for all $r,t$ and $\lambda$ if and only if 
     \begin{enumerate}
         \item $a \leq \alpha_\mu$ or
         \item $\left(\frac{1}{\psi}\right)$ is affine on $(\alpha_\mu,\beta_\mu)$.
     \end{enumerate}
     Indeed, equality in \eqref{eq:half_space_convex} is equivalent to equality in \eqref{eq:concave_iff}, i.e.
     \[
     t\mapsto \frac{\psi^\prime(t)}{\psi(t)^2}
     \]
     is constant on $[\Phi_\mu^{-1}(\Phi_\mu(r)-\Phi_\mu(a)),r]$ for every $r \in (\max\{a,\alpha_\mu\},\beta_\mu)$.
     
     \noindent There are two cases to consider. 
     
     \noindent The first case is that $\Phi_\mu^{-1}(\Phi_\mu(r)-\Phi_\mu(a))=r$, i.e. $\Phi_{\mu}(a)=0$, which means $a \leq \alpha_\mu$.

     \noindent The other case is that $\left(\frac{1}{\psi}\right)$ is affine on this interval. Sending $r \to \max\{a,\alpha_\mu\}$ and $r\to \beta_\mu$ gives $\left(\frac{1}{\psi}\right)$ is affine on $(\alpha_\mu,\beta_\mu)$.

\end{rem}

The function $I_{\mu}$ from \eqref{eq:I_mu} was explored by Bobkov in relation to isoperimetry on the real line \cite{SB96,BH96,BH97_2}.  We are now ready to prove Theorem \ref{t:one_d}.
\begin{proof}[Proof of Theorem~\ref{t:one_d}]
Fix a (potentially unbounded) interval $(a,b)\subset \R$ such that $\mu((a,b))>0$. Let $\psi$ be the density of $\mu$. Define the probability measure $\nu$ on $\R$ via $$d\nu(r)=\frac{1}{\mu((a,b))}\psi(r)\chi_{(a,b)}(r)dr.$$ Notice that $\alpha_\nu = \max\{a,\alpha_\mu\}$,
\[
g_\nu = \int_{\R^n} rd\nu(r) = \frac{1}{\mu((a,b))}\int_{a}^br\psi(r)dr, \quad \text{and} \quad
\nu((\alpha_\nu,g_\nu]) = \frac{\mu((a,g])}{\mu((a,b))}.\]
Also, we have $\Phi_\nu(r)=\frac{1}{\mu((a,b))}(\Phi_\mu(r)-\Phi_\mu(a))$.
From the equivalence of 1.) and 2.) of Lemma~\ref{l:half_space_concave}, the function 
\begin{equation}
\label{eq:convex_function}
r\mapsto \Phi_{\mu}^{-1} \left(\Phi_\mu(r)-\Phi_\mu(a)\right)=\Phi_{\mu}^{-1}\left(\mu((a,b))\Phi_\nu(r)\right)
\end{equation} 
is concave. Therefore, we can apply the first case of Lemma~\ref{l:best}, with $F(r)=\Phi_\mu^{-1}(\mu((a,b))r),$ and obtain
\[
\frac{\mu((a,g])}{\mu((a,b))}\geq \frac{1}{\mu((a,b))}\Phi_{\mu}\left(\int_0^{1}\Phi_\mu^{-1}(\mu((a,b))r)dr \right).
\]
Performing a variable substitution yields the inequality. 
Furthermore, the equality conditions follow from the fact that \eqref{eq:convex_function} is affine if and only if $a\leq \alpha_\mu$.
\end{proof}

In the following proposition, we compute $\int_0^{t} \Phi_{\mu}^{-1}(r)dr$. In particular, it eases any concern about integrability of $\Phi_\mu^{-1}$, and yields an explicit formula that can be used in practice.
\begin{prop}
\label{p:better_right_hand}
Let $\mu$ be a probability measure on $\R$ whose barycenter $g_\mu$ exists and let $t\in [0,1]$.
Then, \[
\int_0^{t} \Phi_{\mu}^{-1}(r)dr = g_{\mu} - \Phi_\mu^{-1}\left(t\right)(1-t).
\]
\end{prop}
\begin{proof}
    Let $u=\Phi^{-1}_\mu(r)$. Then, from integration by parts, we have
    \begin{align*}
    \int_0^{t}\Phi_\mu^{-1}(r)dr &= \int_{\alpha_\mu}^{\Phi_\mu^{-1}(t)}u\Phi_\mu^\prime(u)du
     = [u\Phi_\mu(u)]_{\alpha_\mu}^{\Phi_\mu^{-1}(t)}-\int_{\alpha_\mu}^{\Phi_\mu^{-1}(t)}\Phi_\mu(u)du
    \\
    &=t\Phi_\mu^{-1}(t)-\int_{\alpha_\mu}^{\Phi_\mu^{-1}(t)}\int_{\alpha_\mu}^{u}d\mu(s)du
    =t\Phi_\mu^{-1}(t)-\int_{\alpha_\mu}^{\beta_\mu}\int_{s}^{\Phi_\mu^{-1}(t)}dud\mu(s)
    \\
    &=t\Phi_\mu^{-1}(t)-\int_{\alpha_\mu}^{\beta_\mu}(\Phi_\mu^{-1}(t)-s)d\mu(s) = t\Phi_\mu^{-1}(t)-\Phi_\mu^{-1}(t)+\int_{\alpha_\mu}^{\beta_\mu}sd\mu(s)
    \\
    &=\Phi_\mu^{-1}(t)\left(t-1\right)+g_{\mu}.
    \end{align*}
\end{proof}

\section{Results in higher dimensions}
\label{sec:actual_main}
Having established the one-dimensional Gr\"unbaum-type inequality as a foundation, we now prove our main results. We first show that Lemma~\ref{l:best} in fact implies its analogue in higher dimensions.
\begin{thm}
\label{t:best}
    Suppose $\mu$ is a probability measure on $\R^n$ with density whose barycenter $g_\mu$ exists. Let $H$ a hyperplane through $g_{\mu}$ with outer-unit normal $u$. Let $F:(0,1)\to \R$ be an invertible, continuous function, and suppose that $F$ is increasing (resp. decreasing). 
    
    If the function $r\mapsto (F\circ \mu) \left(\left\{x\in \R^n:\left\langle x,u\right\rangle \leq r\right\}\right)$
    is concave (resp. convex) on its support, then one has the inequality
    \begin{equation}
    \label{eq:lower_half_inequality_measures}
    \mu(H^-)\geq F^{-1}\left(\int_{0}^{1}F(r) dr\right).
    \end{equation}
    There is equality if and only if the associated concave or convex function is, in fact, affine.
\end{thm}

\begin{proof}
Let $K=\supp(\mu)$ and $$K_r=K\cap \{x\in\R^n:\langle x,u \rangle \leq r\}.$$
    Then, $\mu(H^-)=\mu(K_{\langle g_\mu,u \rangle})$. Let $\psi$ be the density of $\mu$. Observe by Fubini's theorem that
    \[
    \mu(K_r)=\int_{K_r}\psi(x)dx=\int_{-h_K(-u)}^r\int_{\{z\in u^\perp:z+tu\in K\}}\psi(z+tu)dzdt.
    \]
    Our hypothesis is precisely that $r\mapsto F\circ \mu(K_r)$ is either concave or convex, depending on the monotonicity of $F$.
    Define the function
    \[
    \psi_1(t)=\int_{\{z\in u^\perp:z+tu\in K\}}\psi(z+tu)dz,
    \]
    and let $\mu_1$ be the Borel measure on $\R$ with density $\psi_1$ and support $(\alpha_\mu,\beta_\mu)$, where $\alpha_\mu=-h_K(-u)$ and $\beta_\mu=h_K(u)$. Notice that $g_{\mu_1} = \langle g_\mu,u \rangle$.  Then, $\Phi_{\mu_1}(r)=\mu(K_r)$, $\mu(H^{-})=\mu_1((\alpha_\mu,g_\mu])$, and $\mu_1(\R)=\mu(\R^n)=1$. Therefore, the claim, and its equality conditions, follows from Lemma~\ref{l:best}.
\end{proof}

Our main goal is to use Theorem~\ref{t:best} to establish Gr\"unbaum-type inequalities for measures with concavity. As a first step we give a proof of Theorem~\ref{t:best_2}. 

\begin{proof}[Proof of Theorem~\ref{t:best_2}]
    Let $\psi$ be the density of $\mu$. Define a probability measure on $\R^n$ via $$d\nu(x)=\frac{1}{\mu(K)}\psi(x)\chi_K(x)dx,$$ and define a function $\tilde F(r)=F(\mu(K)r)$. Then, $\nu(H^-)=\frac{\mu(K\cap H^-)}{\mu(K)}$, $$(F\circ \mu) \left(\left\{x\in K:\left\langle x,u\right\rangle \leq r\right\}\right)=(\tilde F\circ \nu) \left(\left\{x\in \R^n:\left\langle x,u\right\rangle \leq r\right\}\right),$$ and 
    \[
\tilde{F}^{-1}\left(\int_{0}^{1}\tilde{F}(r) dr\right) = \frac{1}{\mu(K)}F^{-1}\left(\int_{0}^{\mu(K)}F(r) \frac{dr}{\mu(K)}\right).
    \]
    Therefore, the inequality and its equality conditions are immediate from Theorem~\ref{t:best}.
\end{proof}

In general, we say a Borel measure $\mu$ is $F$-concave over a collection of convex sets $\mathcal{C}$ if it is absolutely continuous with respect to the Lebesgue measure and there exists a strictly monotone, continuous function $F:[0,\mu(\R^n))\to \R$ such that, for every $\lambda\in [0,1]$ and $K,L\in\mathcal{C}$ it holds
\begin{equation}
    \mu((1-\lambda)K+\lambda L)\geq F^{-1}\left((1-\lambda)F(\mu(K))+\lambda F(\mu(L))\right).
\end{equation}
We will additionally assume that $F\in L_{\text{loc}}^1([0,\mu(\R^n))$. If $\mathcal{C}$ is all convex sets in the support of $\mu$, then we say $\mu$ is $F$-concave without referencing $\mathcal{C}$. Here, $\lambda K=\{\lambda x: x\in K\}$ is the dilate of $K$ by $\lambda$, and $K+L=\{x+y:x\in K,y\in L\}$ is the Minkowski sum of $K$ and $L$.

We first verify that if $\mu$ is a $F$-concave measure with convex support, then its upper-and-lower half-space functions are $F$-concave.
\begin{prop}
\label{p:F_concave}
    Let $\mu$ be a Borel measure with density on $\R^n$. Let $K\subset \R^n$ be a convex set so that $0<\mu(K) <+\infty$. Suppose $\mu$ is $F$-concave, and that $F$ is increasing (resp. decreasing). Then, the function 
    $$r\mapsto F\circ \mu \left(\left\{x\in K:\left\langle x,u\right\rangle \leq r\right\}\right)$$
    is concave (resp. convex).
\end{prop}
\begin{proof}
    Since $K$ is convex, one has that
\begin{equation}
\label{eq:set_inclusion}
K\cap\{\langle x,u\rangle\leq (1-\lambda)a+\lambda b\} \supseteq(1-\lambda)(K\cap\{\langle x,u\rangle \leq a\})+ \lambda(K\cap\{\langle x,u\rangle \leq b\}),
\end{equation}
and
\[K\cap\{\langle x,u\rangle\geq (1-\lambda)a+\lambda b\} \supset(1-\lambda)(K\cap\{\langle x,u\rangle \geq a\})+ \lambda(K\cap\{\langle x,u\rangle \geq b\}).
\]
The claim follows from the $F$-concavity of $\mu$.
\end{proof}

From Proposition~\ref{p:F_concave}, we can give Theorem~\ref{t:best_2} an immediate upgrade.

\begin{thm}
\label{t:F_concave}
    Suppose $\mu$ is a $F$-concave measure. Let $K\subset \R^n$ be a convex set such that $0<\mu(K)<+\infty$ and let $H$ a hyperplane through $g_{\mu}(K)$. Then, it holds
    \begin{equation}
    \label{eq:main_inequality}
   \mu(K\cap H^-)\geq F^{-1}\left(\int_{0}^{t}F(r)\frac{dr}{t}\right), \quad\, t=\mu(K),
    \end{equation}
    There is equality if and only if the function $r\mapsto F\circ \mu \left(\left\{x\in K:\left\langle x,u\right\rangle \leq r\right\}\right)$ is affine on $[-h_K(-u),h_K(u)]$. Here, $u$ is the outer-unit normal of $H^-$. 
\end{thm}

\begin{proof}
    If $\mu$ is $F$-concave, then, by Proposition~\ref{p:F_concave} it satisfies \eqref{eq:lower_half_inequality}.
\end{proof}

\section{The proof of Theorem~\ref{t:s_concave}}
\label{sec:s_concave}
\subsection{On $s$-concave measures and the inequality in Theorem~\ref{t:s_concave}}
\label{s:s_concave_def}
For $a,b>0$ and $\lambda \in [0,1]$ we recall their $\fcon$-mean, for $\fcon\in\R,$ is given by
\begin{align*}
    M^{(\lambda)}_\fcon(a,b)=\begin{cases}
    \min\{a,b\}, & \fcon=-\infty,
    \\
    a^{1-\lambda}b^\lambda, &\fcon =0,
    \\
    \max\{a,b\}, & \fcon =+\infty,
    \\
        \left((1-\lambda)a^\fcon +\lambda b^\fcon \right)^\frac{1}{\fcon}, & \text{otherwise.}
    \end{cases}
\end{align*}
From Jensen's inequality, if $\fcon^\prime>\fcon$, then $M^{(\lambda)}_{\fcon^\prime}(a,b) \geq M^{(\lambda)}_\fcon(a,b)$.
\begin{defn}
\label{def:s_concave}
    We say a Radon measure $\mu$ on $\R^n$ is $s$-concave, $s\in [-\infty,\infty)$ if for every pair of Borel sets $A,B\subset\R^n$ such that $\mu(A)\mu(B)>0$ and $\lambda \in [0,1]$, it holds that
    \begin{equation}
    \label{eq:s_concave}
        \mu((1-\lambda) A+\lambda B)\geq M^{(\lambda)}_s(\mu(A),\mu(B)).
    \end{equation}
\end{defn}
Recall that a measure is Radon if it is a locally finite and inner regular Borel measure. Every Borel measure with density is Radon. Notice that the class of $(-\infty)$-concave measures, also called convex measures, is the largest class. The case $s=0$ is more commonly referred to as log-concavity. The classical Brunn-Minkowski inequality is precisely the statement that the Lebesgue measure on $\R^n$ is $s$-concave, with $s=\frac{1}{n}$, see \textit{e.g.} \cite{G20}.

\begin{defn}
\label{def:kappa_concave}
    We say a function $\psi:\R^n\to [0,\infty)$ is $\fcon$-concave, $\fcon\in\R$, if, for every $x,y\in\R^n$ such that $\psi(x)\psi(y)>0$ and $\lambda\in[0,1]$, one has,
\begin{equation}
    \psi((1-\lambda)x+\lambda y) \geq M^{(\lambda)}_\fcon(\psi(x),\psi(y)).
    \label{eq:kappa_concave}
\end{equation}
\end{defn}
There is a relation between the two definitions, which is precisely Borell's classification of $s$-concave Radon measures.

\begin{cor}[Borell's classification, \cite{Bor75}]
\label{cor:BBL}
    If a function $\psi$ is $\fcon$-concave on $\R^n$, $\fcon \geq -1/n$, then the Borel measure $\mu$ with density $\psi$ is $s$-concave, with $s=\frac{\fcon}{1+n\fcon}$. Conversely, if $\mu$ is a $s$-concave Radon measure on $\R^n$, then $s\leq \frac{1}{d}$, where $d\in\{0,\dots,n\}$ is the dimension of the support of $\mu$ and it has density that is $p$-concave with respect to the Hausdorff measure on the affine subspace containing by its support, with $p=\frac{s}{1-ds}$. 
\end{cor}

We include a non-exhaustive list of investigations of $s$-concave measures in convex geometry \cite{GAL19,LRZ22,LP25,AHNRZ21,MR14,AAKM04,CFGLSW16,MSZ18,LXS24}, probability theory \cite{AGLLOPT12,BS07,FLM20,BM12}, and their intersection \cite{LYZ04_3,LLYZ12,LLYZ13,FG06,CEFPP15}. Now that we have introduced definitions, we can establish the inequalities in Theorem~\ref{t:s_concave}.
\begin{proof}[Proof of the inequalities in Theorem~\ref{t:s_concave}]
    Recall that a $s$-concave Radon measure with $s>-1$ has a finite first moment, i.e. $g_\mu$ exists. If the dimension $d$ of the affine subspace $E$ containing the support of $\mu$ is $0$, then $\mu$ is a Dirac mass at its barycenter. Thus, $\mu(H^{-})=1$ and the inequality is trivial. 
    
    If $d\in \{1,\dots,n\}$, then, by Corollary~\ref{cor:BBL}, $\mu$ has a density $\psi$ that is $\fcon=\left(\frac{s}{1-ds}\right)$-concave. The inequalities \eqref{eq:Grun_s_neq_0} and \eqref{Gruabaum inequality for log-concave} then follow from the fact that Proposition~\ref{p:F_concave}, with $K=\supp(\mu)$, reveals that $\mu$ satisfies Theorem~\ref{t:best}, specifically equation~\eqref{eq:lower_half_inequality_measures}, with $F(r)=r^s$ for $s\neq 0$ (yielding \eqref{eq:Grun_s_neq_0}) and $F(r)=\log r$ for $s=0$ (yielding \eqref{Gruabaum inequality for log-concave}). 
\end{proof}

In contrast, the equality conditions require much more work. This is the aim of the remainder of this section.

\subsection{Equality characterization: preparation}
We start by recalling the Borell-Brascamp-Lieb inequality \cite{BrasLieb2}: Fix $\lambda \in (0,1)$ and let $\fcon \geq -\frac{1}{d}$. Suppose $h,g,f:\R^d\to\R_+$ are a triple of integrable functions  satisfying
\[
h((1-\lambda) x+\lambda y)\geq M^{(\lambda)}_\fcon(g(y),f(x)).
\]
Then, by setting $s=\frac{\fcon}{1+d\fcon}$,
\begin{equation}
\label{eq:BBL}
\int_{\R^d}h(x) dx \geq M^{(\lambda)}_s\left(\int_{\R^d}g(x) dx,\int_{\R^d}f(x) dx\right).
\end{equation}

S. Dubuc \cite{SD77} later established the equality case: there is equality in \eqref{eq:BBL} if and only if there exists $m>0$ and $b\in\R^d$ such that, for almost all $x\in \R^d$
\begin{equation}
\label{eq:Dubuc}
    \frac{f(x)}{\int_{\R^d}f(x)dx}=g(mx+b)\frac{m^d}{\int_{\R^d}g(x)dx}=h((1-\lambda)(mx+b)+\lambda x)\frac{((1-\lambda)m+\lambda)^d}{\int_{\R^d}h(x)dx}.
\end{equation}

The forward direction of Corollary~\ref{cor:BBL} actually follows from \eqref{eq:BBL} by setting, for $K$ and $L$ convex sets in $\R^n$ and $\psi$ a $\fcon$-concave function on an affine subspace $E$ of $\R^n$ of dimension $d$ (which can be identified with $\R^d$),
\begin{equation}
\label{eq:Dubuc_choice}
\begin{split}
h((1-\lambda) x+\lambda y)=(\chi_{(1-\lambda) K+\lambda L}\cdot \psi)((1-\lambda) x+\lambda y),
\\
 f(x)=\chi_{L}(x)\psi(x), \quad \text{and} \quad g(y)=\chi_K(y)\psi(y).
 \end{split}
\end{equation}
In the next proposition, we show what \eqref{eq:Dubuc} becomes with this choice of functions.
\begin{prop}
\label{p:equality_dubuc}
    Let $\mu$ be a $s$-concave Radon measure supported on an affine subspace $E$ of $\R^n$ with dimension $d\in \{1,\dots,n\}$. By Corollary~\ref{cor:BBL}, it has a density $\psi$ on $E$ that is $\fcon=\left(\frac{s}{1-ds}\right) \geq -\frac{1}{d}$ concave. 
    
    Let $\lambda \in (0,1)$ and suppose $K,L\subset \supp(\mu)$ are convex sets such that
    \begin{equation}
    \label{eq:equality_s}
    \mu((1-\lambda)K + \lambda L)^s = (1-\lambda)\mu(K)^s+\lambda \mu(L)^s.
    \end{equation}
    Then, there exists $m>0$ and $b\in E$ such that
    \begin{enumerate}
        \item  $K=mL+b$ up to $\mu$ null sets;
        \item  $\psi(x)=\left(\frac{\mu(L)}{\mu(K)}m^d\right)\psi(mx+b)$ for almost every $x\in L$;
        \item $\psi((1-\lambda)(mx+b)+\lambda x)=M_\fcon^{(\lambda)}\left(\psi(mx+b),\psi(x)\right)$ for almost every $x\in L$.
    \end{enumerate}
\end{prop}
The Proposition~\ref{p:equality_dubuc} was previously explored by E. Milman and L. Rotem \cite{MR14}. Our new contribution is allowing $d\neq n$ and, more importantly, the third item. We will need the following well-known improvement of the arithmetic-geometric mean inequality: suppose $\alpha_1,\alpha_2\in \R$ are such that $\alpha_1+\alpha_2>0$ and define $\alpha=\frac{\alpha_1\alpha_2}{\alpha_1+\alpha_2}$. Then, for $\lambda\in(0,1)$ and $a_1,a_2,b_1,b_2>0$, we have
\begin{equation}
    M_{\alpha_1}^{(\lambda)}(a_1,b_1)M_{\alpha_2}^{(\lambda)}(a_2,b_2) \geq M_{\alpha}^{(\lambda)}(a_1a_2,b_1b_2).
    \label{eq:best_am_gm}
\end{equation}
\begin{proof}[Proof of Proposition~\ref{p:equality_dubuc}]
    The hypothesis \eqref{eq:equality_s} is nothing but the assertion that there is equality in the Borell-Brascamp-Lieb inequality, \eqref{eq:BBL}, used in $E$ for the choice of functions \eqref{eq:Dubuc_choice}.
    From the first equality in \eqref{eq:Dubuc}, we see that there exist $m>0$ and $b\in E$ such that, for almost every $x\in L$,
    \[
    \psi(x)\chi_L(x)=\left(\frac{\mu(L)}{\mu(K)}m^d\right)\psi(mx+b)\chi_{K}(mx+b).
    \]
    This yields items 1 and 2.

    As for item 3, we now use the second equality in \eqref{eq:Dubuc} and obtain, for every $x\in L$,
    \begin{align*}
    \frac{\psi(x)}{\mu(L)}&=\frac{((1-\lambda)m+\lambda)^d}{\mu((1-\lambda)K + \lambda L)}\psi((1-\lambda)(mx+b)+\lambda x)
    \\
    &\geq \frac{((1-\lambda)m+\lambda)^d}{\mu((1-\lambda)K + \lambda L)}M^{(\lambda)}_\fcon(\psi(mx+b),\psi(x))
    \\
    &=\frac{((1-\lambda)m+\lambda)^d}{\mu((1-\lambda)K + \lambda L)}M^{(\lambda)}_\fcon\left(\left(\frac{\mu(K)}{m^d}\right)\frac{\psi(x)}{\mu(L)},\mu(L)\frac{\psi(x)}{\mu(L)}\right),
    \end{align*}
    where we used that $\psi$ is $\fcon$-concave and item 2. We see that every term has $\frac{\psi(x)}{\mu(L)}$, and thus these terms cancel.
    We deduce that
\begin{align*}
    \mu((1-\lambda)K + \lambda L)&\geq ((1-\lambda)m+\lambda)^dM^{(\lambda)}_\fcon\left(\left(\frac{\mu(K)}{m^d}\right),\mu(L)\right)
    \\
    &=M_d^{(\lambda)}(m^d,1)M^{(\lambda)}_\fcon\left(\left(\frac{\mu(K)}{m^d}\right),\mu(L)\right)
    \\
    &\geq M^{(\lambda)}_s\left(\mu(K),\mu(L)\right),
    \end{align*}
where we used \eqref{eq:best_am_gm}. However, from \eqref{eq:equality_s}, we know there is equality. Therefore, there is equality in our use of the $\fcon$-concavity of $\psi$, and item 3 follows.    
\end{proof}

Finally, we need the following application of Fubini's theorem.
\begin{lem}
\label{l:change_variables_non_ortho}
    Let $\sigma$ be a signed measure on $\R^n$. Let $u\in \s^{n-1}$ and let $v\in\R^n$ be a vector so that $\langle v,u\rangle >0$. Suppose that $\sigma$ has a density $\psi$ such that, for every $z\in u^\perp$ and $r \in \R$, it has the following product decomposition: $\psi(z+rv)=A(r)w(z)$ for some function $A$ on $\R$ and $w$ on $u^\perp$. Then, for every Borel set $K\subset \R^n$,
    \begin{equation}
        \sigma\left(\left\{x\in K:\langle x,u\rangle \leq r\right\}\right)=\langle v,u \rangle \int_{-\infty}^{\frac{r}{\langle v,u \rangle}}A(t)\int_{\{z\in u^\perp: z+t  v\in K\}}w(z)dzdt.
        \label{eq:change_variables_non_ortho}
    \end{equation}
\end{lem}
\begin{proof}
Let $K_r=K\cap \{x\in\R^n:\langle x,u \rangle \leq r\}$. Then, from Fubini's theorem, we have
    \begin{align*}
    \sigma(K_r) = \int_{K_r}\psi(x)dx=\int_{-\infty}^{r}\int_{\{y\in u^\perp: y+\ell u\in K\}}\psi(y+\ell u)dyd\ell.
\end{align*}
Now, we perform a change of basis. On the one hand, we can write $x=y+\ell  u$, where $y\in u^\perp$ and $\ell = \langle x,u\rangle\in \R$. On the other-hand, we can write $x=z+t v$, where $t\in \R$. By further decomposing $v=q+ \langle v,u \rangle u$, with $q\in u^\perp$, we have $x=y+t q+t \langle v, u\rangle u$, which shows $t=\frac{\ell}{\langle v,u \rangle}$ and $y=z-t q=z-\frac{\ell}{\langle v,u \rangle }q$. Therefore,
\begin{align*}
    \sigma(K_r) &=\langle v,u \rangle \int_{-\infty}^{\frac{r}{\langle v,u \rangle}}\int_{\{z\in u^\perp: z+t  v\in K\}}\psi(z+t v)dzdt
    \\
    &=\langle v,u \rangle \int_{-\infty}^{\frac{r}{\langle v,u \rangle}}A(t)\int_{\{z\in u^\perp: z+t  v\in K\}}w(z)dzdt,
\end{align*}
as claimed.
\end{proof}

We will make use of the recession function $W^+$ of a proper convex function $W$. We recall both its definition and a relevant fact from the book of Rockafellar. Note that it also arises as a limit of the so-called perspective function of $W$, which is the function on $\R^{n+1}$ given by $(x,t)\mapsto tW\left(\frac{x}{t}\right)$. Recall $W$ is said to be proper if $W(x) >-\infty$ for all $x\in\R^n$, and that the domain of $W$ is 
\[
\operatorname{dom}(W)=\{x\in\R^n:W(x)<+\infty\}.
\]
\begin{defn}
\label{def:recession}
    Let $W$ be a proper, lower semi-continuous, convex function. Then, its recession function $W^+$ is a positively $1$-homogeneous function given by, for $y\in \R^n$,
    \[
    W^+(y)=\sup_{x \in \operatorname{dom}(W)}\left(f(x+y)-f(x)\right).
    \]
    Furthermore \cite[Theorem 8.5]{RTR70}, for every $x\in \operatorname{dom}(W)$ and $y\in \R^n$, one has that the limit
\begin{equation}
\label{eq:recession}
    \lim_{t\to +\infty}\frac{W(x+t y)-W(x)}{t}
\end{equation}
exists in $\R\cup\{+\infty\}$, is independent of $x$, and equals $W^+(y)$.
\end{defn}

 \subsection{Equality characterization: shared facts}
For equality cases, it is straightforward to verify the sufficiency of the formulas \eqref{eq:psi_forumla_pos}, \eqref{eq:psi_affine_s_0_intro}, and \eqref{eq:psi_affine_neg_intro} using the Lemma~\ref{l:change_variables_non_ortho}. At the end of Section~\ref{sec:s_log_equal}, we will demonstrate for the reader the $s=0$ case. It remains to show the necessity, i.e. that equality in Gr\"unbaum's inequality for $\mu$ implies the claimed formulas for its density $\psi$.

Our strategy is, in broad strokes, the following: We first apply the consequences of the Borell-Brascamp-Lieb inequality, in the form of Proposition~\ref{p:equality_dubuc}, to $K$, the support of $\mu$. Then, we use this information to construct $\psi$, the density of $\mu$. Moreover, for the reader's convenience, we split the proof depending on whether $s$ is positive, negative, or $s=0$. Nevertheless, we start by showing certain facts shared by the three cases. 

Without loss of generality, we may assume that the affine subspace $E$ containing the support of $\mu$ is $\R^n$. Otherwise, if the dimension of $E$ is $d$, then we may identify $E$ with $\R^d$ via an isometry and work in $\R^d$. Also, our density $\psi$, being $p$-concave, is upper semi-continuous on the interior of its support. Since we can change $\psi$ on a set of measure zero without changing $\mu$, we may assume that $\psi$ is upper semi-continuous on its entire support.

First, we introduce the notation $$K_r=K\cap \{\langle x,u \rangle \leq r \}.$$ 
Defining $r_0=-h_K(-u)$ and $r_1=h_K(u)$, then $K_{r_1}=K$ and $$r\mapsto \mu(K_r)$$ is supported on $[r_0,r_1]$. We note that since $K$ is not necessarily compact, the numbers $r_0$ and $r_1$ are not \textit{a priori} finite. Finally, recall that $\mu$ is probability and so $\mu(K)=1$.

Suppose there is equality in Theorem~\ref{t:s_concave} for $H=g_\mu + u^\perp$. Then, either $\mu(K_r)^s$ (if $s\neq 0$) or $\log \mu(K_r)$ (if $s=0$), is affine. Since $\mu(K_{r_1})^s=1$ ($s\neq 0$) or $\log \mu(K_{r_1})=0$ ($s=0$), one has $r_1<+\infty$. Also, since it is increasing on $[r_0,r_1]$, we deduce that there exists $c>0$ such that, for $r\in [r_0,r_1]$,
\begin{equation}
\label{eq:Q_s}
\mu \left(K_r\right)=
\begin{cases}
    (1+cs(r-r_1))^\frac{1}{s} & s \neq 0,
    \\
    e^{c(r-r_1)}, &s=0.
\end{cases}\end{equation}
We then obtain
\begin{equation}
\label{eq:q_s}
\frac{d}{dr}\mu(K_r)=\begin{cases}
    c(1+cs(r-r_1))^\frac{1-s}{s}, & s\neq 0,
    \\
    ce^{c(r-r_1)}, & s=0.
\end{cases}
\end{equation}
Our notation suppresses the fact that $c$ is also a positive function of $s$, as this is immaterial for our considerations.
Next, we recall that \eqref{eq:set_inclusion} yields, for every $\rho,\tau\in (r_0,r_1]$ and $\lambda \in (0,1)$,
\[
K_{(1-\lambda)\rho+\lambda \tau}\supseteq (1-\lambda)K_{\rho} + \lambda K_\tau.
\]
Therefore,
\begin{equation}
\label{eq:actually_using_s_concave}
    \mu(K_{(1-\lambda)\rho+\lambda \tau}) \geq \mu((1-\lambda)K_{\rho} + \lambda K_\tau)
    \geq M_s^{(\lambda)}\left(\mu(K_{\rho}), \mu(K_\tau)\right).
\end{equation}
However, by \eqref{eq:Q_s}, there is equality throughout. 
Thus, 
\begin{equation}
\label{eq:almost_cone}
K_{(1-\lambda)\rho+\lambda \tau} = (1-\lambda)K_{\rho} + \lambda K_\tau.
\end{equation}
This implies that for any $r_0< \rho<\tau \leq r_1$ one has 
\[
K\cap\{x: \rho\le\langle x,u\rangle\le\tau\}=\conv(K\cap\{x:\langle x,u\rangle=\rho\},K\cap\{x:\langle x,u\rangle=\tau\}).
\]
And, by taking $\tau=r_1$ and using Proposition~\ref{p:equality_dubuc}, there exists $m_\rho>0$ and $b_\rho\in\R^n$ such that
\begin{equation}
\label{eq:almost_first_psi_affine}
K_{\rho} = m_\rho K+b_\rho \quad \text{and} \quad \psi(x)=\frac{m_\rho^n}{\mu(K_\rho)}\psi(m_\rho x+b_\rho),\end{equation}
and
\begin{equation}
\label{eq:first_psi_affine}
    \psi((1-\lambda)(m_\rho x+b_\rho)+\lambda x)=M_\fcon^{(\lambda)}\left(\psi(m_\rho x+b_\rho), \psi(x)\right).
\end{equation}
for every $x\in K$ and $\lambda \in [0,1]$.

Henceforth, we set $$L=K \cap \{x:\langle x,u\rangle=r_1\}.$$ Let $x \in K_\rho \cap \{x:\langle x,u\rangle=\rho\}$. Then, since an affine map will map the boundary of one convex set to another, one has by \eqref{eq:almost_first_psi_affine} that there exists $z\in L$ such that $x= m_\rho z + b_\rho$. Taking the inner-product of both sides with $u$, we obtain
\begin{equation}
    \rho = m_\rho \cdot r_1 +\langle b_\rho,u \rangle.
    \label{eq:rho_relations}
\end{equation}

As an immediate consequence of this, we obtain from \eqref{eq:almost_first_psi_affine} and \eqref{eq:almost_cone} the formula
\begin{equation} 
\label{eq:s_sets}
\begin{split}
m_{(1-\lambda)\rho+\lambda r_1}K+b_{(1-\lambda)\rho+\lambda r_1}&=K_{(1-\lambda)\rho+\lambda r_1} 
\\
&= (1-\lambda)K_{\rho} + \lambda K
\\
&= ((1-\lambda)m_\rho + \lambda)K+(1-\lambda)b_\rho.
\end{split}
\end{equation}
Fixing arbitrary $\xi \in (r_0,\rho)$, we deduce by first taking intersections of the first and last sets in the equality in \eqref{eq:s_sets} with $\{x:\langle x,u \rangle=\xi\}$ and then taking the inner-product of their elements with $u$, that
\begin{align*}
m_{(1-\lambda)\rho+\lambda r_1}\cdot \xi+\langle b_{(1-\lambda)\rho+\lambda r_1},u \rangle=((1-\lambda)m_\rho + \lambda)\cdot \xi+(1-\lambda)\langle b_\rho,u\rangle.
\end{align*}
By using \eqref{eq:rho_relations}, and defining $m_{r_1}:=1$, we then obtain a quantity independent of $\xi$:
\begin{equation}
\label{eq:ms}
m_{(1-\lambda)\rho+\lambda r_1} =(1-\lambda)m_\rho +\lambda m_{r_1}.
\end{equation}
The formula \eqref{eq:ms} shows that $\rho\mapsto m_\rho$ is affine on $(r_0,r_1)$. Since $m_{r_1}=1$, we have there exists $\zeta(s)$ such that
\begin{equation}
    m_\rho=(\zeta(s)(r_1-\rho)+1)_+,\quad \rho \leq r_1.
    \label{eq:m_is_linear}
\end{equation}

Inserting \eqref{eq:ms} into \eqref{eq:s_sets}, we have
\begin{equation}
\label{eq:almost_cylinder}
\begin{split}
((1-\lambda)m_\rho +\lambda m_{r_1})K+b_{(1-\lambda)\rho+\lambda r_1} = ((1-\lambda)m_\rho +\lambda m_{r_1})K+(1-\lambda)b_{\rho}.
\end{split}
\end{equation}
By defining $b_{r_1}:=0$, we obtain 
\begin{equation}
\label{eq:rho_xi_relate}
b_{(1-\lambda)\rho+\lambda r_1} = (1-\lambda)b_\rho + \lambda b_{r_1}. 
\end{equation}
We define the vector 
\begin{equation}
\label{eq:v_def}
    v = \frac{1}{\langle b_\rho,u\rangle}b_\rho = \frac{b_\rho}{\rho-m_\rho r_1},
\end{equation}
where the second equality is from \eqref{eq:rho_relations}. We observe that \eqref{eq:rho_xi_relate} and \eqref{eq:rho_relations} yield that $v$ is independent of the choice of $\rho$, i.e. \eqref{eq:v_def} holds for every $\rho\in (r_0,r_1)$. Inserting \eqref{eq:v_def} into \eqref{eq:almost_first_psi_affine}, we have shown that, for all $\rho \in (r_0,r_1)$,
\begin{equation}
K_\rho=m_\rho K + \left(\rho-r_1\right)\left(1+\frac{m_\rho r_1-r_1}{r_1-\rho}\right)v.
\label{eq:updated_set_relations}
\end{equation}
 We next eliminate $m_\rho$: by inserting \eqref{eq:m_is_linear} into \eqref{eq:updated_set_relations}, we obtain
\begin{equation}
K_\rho=K+\zeta(s)(r_1-\rho) (K-r_1v) +(\rho-r_1)v.
\label{eq:updated_set_relations_3}
\end{equation}

The next step is to compute $\zeta(s)$. To this aim, we integrate both sides of \eqref{eq:almost_first_psi_affine} over $\{x\in \R^n: \rho \leq \langle x,u \rangle \leq r_1\}$ to obtain
\begin{align*}
    1-\mu(K_\rho) &= \frac{m_\rho^n}{\mu(K_\rho)}\int_{\{x\in \R^n: \rho \leq \langle x,u \rangle \leq r_1\}}\psi(m_\rho x+b_\rho)dx
    \\
    &=\frac{1}{\mu(K_\rho)} \int_{\{x\in \R^n: (m_\rho +1)\rho-m_{\rho}r_1 \leq \langle x,u\rangle \leq \rho\}}\psi(x)dx
    \\
    &= \frac{1}{\mu(K_\rho)} \left(\mu(K_\rho)-\mu\left(K_{(m_\rho +1)\rho-m_{\rho}r_1}\right)\right),
\end{align*}
where, in the second line, we used performed a variable substitution and used \eqref{eq:rho_relations} to re-write the set of integration. Therefore, we have
\begin{equation}
\label{eq:log_equal}
\mu(K_\rho) = \mu(K)^\frac{1}{2}\mu\left(K_{(m_\rho +1)\rho-m_{\rho}r_1}\right)^\frac{1}{2}.
\end{equation}

We can now re-write the density of $\mu$. By using \eqref{eq:v_def}, the identity \eqref{eq:first_psi_affine} rewrites as 
\begin{equation}
\label{eq:second_psi_affine}
    \psi((1-\lambda)(m_\rho (y-r_1v)+\rho v)+\lambda y)=M_\fcon^{(\lambda)}\left(\psi(m_\rho (y-r_1v)+\rho v), \psi(y)\right)
\end{equation}
for every $y\in L$ and $\lambda \in [0,1]$. We now decompose with respect to $u^\perp$ and $\spa (v)$: write, for $z\in u^\perp$ and $r\in \R$ that
\[(1-\lambda)(m_\rho (y-r_1v)+\rho v)+\lambda y = z+rv.\]
Then, supposing that $\rho<r$, we have
\[
\lambda = \frac{r-\rho}{r_1-\rho}, \quad 1-\lambda=\frac{r_1-r}{r_1-\rho}, \quad \text{and} \quad z=\frac{y-r_1v}{r_1-\rho}\left((r-\rho)+(r_1-r)m_\rho\right).
\]
Using \eqref{eq:m_is_linear}, we have, for $\rho<r<r_1$, that
\[
(r-\rho)+(r_1-r)m_\rho=(r_1-\rho)((r_1-r)\zeta(s)+1)=(r_1-\rho)m_r.
\]

Introducing these parameters into \eqref{eq:second_psi_affine}, we obtain, for $\rho<r<r_1$,
\begin{equation}
\label{eq:third_psi_affine}
\begin{split}
    \psi(z+rv)=M_{\fcon}^{( \frac{r-\rho}{r_1-\rho})}\left(\psi\left(\frac{m_\rho }{m_r}z+\rho v\right), \psi\left(\frac{1}{m_r}z+r_1v\right)\right).
\end{split}
\end{equation}
Now that we have laid out our foundational tools, we specialize to different cases of $s$. 

\subsection{Equality characterization: the case when $s>0$}

We first suppose $s>0$. Since $\mu(K_r)\ge0$ for $r\in [r_0,r_1]$, we have $r_0>-\infty$. And since $K_{r_0}$ has empty interior, one has $\mu(K_{r_0})=0$. 
Using \eqref{eq:Q_s}, we deduce that $c=\frac{1}{s(r_1-r_0)}$. Thus, 
\begin{equation}
\label{eq:mu_s_affine_formula}
\mu \left(K_r\right)=\left(\frac{r-r_0}{r_1-r_0}\right)^\frac{1}{s}.
\end{equation}

We next argue that $K_{r_0}$ is a singleton. Since $K_{r_0}$ has empty interior, $\vol_n(K_{r_0})=0$.
Since $K_\rho$ and $K$ are convex bodies, we can take volume throughout \eqref{eq:updated_set_relations} and obtain
\begin{equation} m_\rho=\left(\frac{\vol_n(K_\rho)}{\vol_n(K)}\right)^\frac{1}{n}.
\end{equation}
From the fact that $K_{\rho} \to K_{r_{0}}$ in the Hausdorff metric as $\rho\to r_0$, we obtain 
\[
m_{r_0} = \lim_{\rho\to r_0} m_\rho = \lim_{\rho\to r_0}\left(\frac{\vol_n(K_\rho)}{\vol_n(K)}\right)^\frac{1}{n}=\left(\frac{\vol_n(K_{r_0})}{\vol_n(K)}\right)^\frac{1}{n}=0.
\]
Consequently, $m_{r_0}=0$. We then obtain $\zeta(s)=-\frac{1}{r_1-r_0}$, that is \eqref{eq:m_is_linear} becomes
\begin{equation}
\label{eq:m_rho_pos}
m_\rho=\left(\frac{\rho-r_0}{r_1-r_0}\right)_+,\quad \rho \leq r_1.
\end{equation}

Then, by viewing the equality of sets in \eqref{eq:almost_first_psi_affine} in terms of the equality of their support functions point-wise on the sphere, we deduce the set $\{b_\rho:{\rho = r_0+\frac{1}{k}(r_1-r_0), k\geq 2}\}$ has an accumulation point as $k\to \infty$, which we call $b$. Therefore,
\[
K_{r_{0}} = \lim_{\rho\to r_0}\left(m_{\rho}K+b_{\rho}\right)=\lim_{\rho\to r_0}b_{\rho}=\lim_{k\to\infty}b_{r_0+\frac{1}{k}(r_1-r_0)}=b.
\]
This shows that $K_{r_0}$ is the singleton $b$. We note that, by sending $\rho\to r_0$ in \eqref{eq:v_def}, we have $v=\frac{b}{\langle b ,u \rangle}$. Sending $\rho\to r_0$ in \eqref{eq:almost_cone} then yields
\begin{equation}
\label{eq:almost_cone_2}
K_{(1-\lambda)r_0+\lambda r_1} = (1-\lambda)b + \lambda K.
\end{equation} 
We obtain from \eqref{eq:almost_cone_2} and \eqref{eq:s_sets} that, for every $\rho\in [r_0,r_1]$, one has
\begin{equation}
\label{eq:almost_cone_2.5}
m_\rho K + b_\rho = K_{\rho} = (1-\lambda)b + \lambda K, \quad \rho=(1-\lambda)r_0 + \lambda r_1.
\end{equation} 

From here, it is clear to see that $b_{\rho}$ and $b$ are linearly related - by taking intersections of the left and right of \eqref{eq:almost_cone_2.5} with $\{x\in\R^n:\langle x,u \rangle=r_0\}$, we obtain \begin{equation}b_{\rho} = (1-m_{\rho})b.
\label{eq:vertex_relate}
\end{equation}

Let $L=K\cap \{x\in \R^n:\langle x,u \rangle=r_1\}$. We now claim that
$$K=\conv(L,\{b\}),$$
Clearly, $\conv(L,\{b\})\subseteq K$, and, for the reverse inclusion, write $$K=\cup_{\rho\in [r_0,r_1]}\left\{ x\in K:\langle x,u\rangle = \rho \right\}.$$ It suffices to show that $\left\{ x\in K:\langle x,u\rangle = \rho \right\} \subset \conv(L,\{b\})$ for all $\rho\in [r_0,r_1]$. This follows from taking intersections in \eqref{eq:almost_cone_2}:
\begin{equation}
\label{eq:almost_cone_3}
\begin{split}
\{x\in K : \langle x,u \rangle=(1-\lambda)r_0 + \lambda r_1\}
&=K_{(1-\lambda)r_0+\lambda r_1}\cap \{\langle x,u \rangle=(1-\lambda)r_0+\lambda r_1\}
\\
&= ((1-\lambda)b + \lambda K)\cap\ \{\langle x,u \rangle=(1-\lambda)r_0+\lambda r_1\}
\\
&= (1-\lambda)b+ \lambda  (K\cap\{\langle x,u \rangle=r_1\}).
\end{split}
\end{equation}

We now focus on the density of $\mu$ and the value of $s$. Taking limits as $\rho \to r_0$, we obtain from \eqref{eq:first_psi_affine}
\begin{equation}
\label{eq:first_psi_affine_again}
    \psi((1-\lambda)b+\lambda x)^\fcon=(1-\lambda)\psi(b)^\fcon+\lambda \psi(x)^\fcon.
\end{equation}
for every $x\in K$ and $\lambda\in [0,1]$. We note that, by taking $x \in L$, \eqref{eq:first_psi_affine_again} yields that $\psi$ is $\fcon$ affine when restricted to the segments in $K$ that start at $b$ and end at a point in $L$. Next, we insert \eqref{eq:vertex_relate} into the second equation in \eqref{eq:almost_first_psi_affine} and obtain
\begin{equation}
\label{eq:almost_psi_constant}
\psi(x)=\left(\frac{m_\rho^n}{\mu(K_\rho)}\right)\psi(m_\rho(x-b)+b), \quad x\in K.
\end{equation}
Evaluating at $x=b$, we obtain
\[
\psi(b)=\left(\frac{m_\rho^n}{\mu(K_\rho)}\right)\psi(b)
\]
We must have either $\psi(b)=0$ or $m_\rho^n = \mu(K_\rho)$ for all $\rho\in (r_0,r_1]$. We now show that the second case implies that $\psi$ is constant on $K$. From \eqref{eq:m_rho_pos}, we have $m_\rho < 1$. Then, combining \eqref{eq:almost_psi_constant} with \eqref{eq:first_psi_affine_again}, we obtain
\begin{align*}
    \psi(x)^\fcon = \psi((1-m_\rho)b+m_\rho x)^\fcon = (1-m_\rho)\psi(b)^\fcon+m_\rho\psi(x)^\fcon,
\end{align*}
which yields $\psi(x)=\psi(b)$ for every $x\in K$. We remark also that, since $\mu(K_\rho)$ is $s$-affine and $m_\rho$ is affine, the identity $m_\rho^n = \mu(K_\rho)$ forces $s=\frac{1}{n}$.

Next, we derive the explicit formula for $\psi$. We define $w_1\equiv \psi_{|_L}$. Then, \eqref{eq:first_psi_affine_again} re-writes as, for $y \in L$,
\begin{equation}
\label{eq:psi_forumla_pos_1}
\psi((1-\lambda)b+\lambda y)=\lambda^\frac{1}{\fcon}w_1(y).
\end{equation}
Notice that every vector $x\in\R^n$ can be written as $x=z+rv$ for some $z\in u^\perp$ and $r>0$. 
Writing
\[
z+rv = (1-\lambda)b+\lambda y,
\]
we have
\[
\lambda=\frac{r-r_0}{r_1-r_0} \quad \text{and} \quad 1-\lambda = \frac{r_1-r}{r_1-r_0}.
\]
Therefore,
\[
y=\frac{r_1-r_0}{r-r_0}(z+rv)+\frac{r-r_1}{r-r_0}b.
\]
Consequently, we can write as \eqref{eq:psi_forumla_pos_1}
\begin{equation}
\label{eq:psi_forumla_pos_2}
\psi(z+rv)=\left(\frac{r-r_0}{r_1-r_0}\right)_+^\frac{1}{\fcon}w_1\left(\frac{r_1-r_0}{r-r_0}(z+rv)+\frac{r-r_1}{r-r_0}b\right)\chi_{(-\infty,r_1]}(r), \quad z\in u^\perp, \; r\in \R.
\end{equation}
Next, inserting $b=r_0v$ \eqref{eq:psi_forumla_pos_2} becomes
\begin{equation}
\label{eq:psi_forumla_pos_3}
\psi(z+rv)=\left(\frac{r-r_0}{r_1-r_0}\right)_+^\frac{1}{\fcon}w_1\left(\left(\frac{r_1-r_0}{r-r_0}\right)z+r_1v\right)\chi_{(-\infty,r_1]}(r), \quad z\in u^\perp, \; r\in \R.
\end{equation}
By defining $w$ on $u^\perp$ via $w(y)=w_1(y+r_1v)$ and $a=(r_1-r_0)^{-1}$, we obtain the formula \eqref{eq:psi_forumla_pos}.

\subsection{Equality characterization: the case of log-concave measures}
\label{sec:s_log_equal}

Now, we move onto the case when $s=0$. We first show the necessity. By integrating \eqref{eq:q_s} over $[r_0,r_1]$ and then comparing to \eqref{eq:Q_s}, we deduce that $r_0=-\infty$. Thus, $K$ is an unbounded convex set. 
Next, we claim that $m_\rho \equiv 1$. Indeed, from the fact that $\mu(K_r)$ is log-affine, we obtain from \eqref{eq:log_equal} that, for every $\rho< r_1$,
\[
\mu(K_{\frac{1}{2}\left((m_\rho+1)\rho-m_\rho r_1\right)) + \frac{1}{2}r_1})=\mu(K)^\frac{1}{2}\mu(K_{(m_\rho+1)\rho-m_\rho r_1})^\frac{1}{2}=\mu(K_\rho)
\]
Therefore,
\[
\rho = \frac{1}{2}\left((m_\rho+1)\rho-m_\rho r_1\right)) + \frac{1}{2}r_1.
\]
Re-arranging, you get $(m_\rho-1)\rho=(m_\rho-1)r_1$. Since $\rho <r_1$, we must have $m_\rho=1$. In terms of $\zeta(s)$ from \eqref{eq:m_is_linear}, this is $\zeta(s)=0$.

Inserting this fact into the formula \eqref{eq:updated_set_relations_3}, we get
\begin{equation}
\label{eq:cylinder_0}
K_\rho = K-(r_1-\rho)v.
\end{equation}
We conclude this case by showing that \eqref{eq:cylinder_0} yields
\begin{equation}
    K= L +(-\pos(v)).
    \label{eq:cylinder}
\end{equation}
Indeed, if $x \in L + (- \pos(v))$, then $x = y-tv$, where $y\in L$ and $t \geq 0$. Define $\rho :=r_1-t$, to deduce $x = y-(r_1-\rho)v\in K-(r_1-\rho)v = K_\rho \subseteq K$. For the other direction, we write $$K=\cup_{\rho \leq r_1}\left\{ x\in K:\langle x,u\rangle = \rho \right\}.$$ Similar to the $s>0$ case, it suffices to show that $\left\{ x\in K:\langle x,u\rangle = \rho \right\} = L - (r_1-\rho)v$ for all $r \leq r_1$. This then follows from  taking intersections of \eqref{eq:cylinder_0}:
\begin{equation}
\label{eq:cylinder_final}
\begin{split}
\{x\in K : \langle x,u \rangle=\rho \}&= (K-(r_1-\rho)v)\cap\ \{\langle x,u \rangle= \rho\}
\\
&= \left\{ y-(r_1-\rho)v:\langle y,u \rangle=r_1\right\} = L-(r_1-\rho)v.
\end{split}
\end{equation}

Now we deduce the density of $\mu$.  Using that $m_\rho =m_r=1$ and $\fcon=0$, \eqref{eq:third_psi_affine} becomes
\begin{equation}
    \label{eq:log_affine}
    \psi(z+rv)=\psi(z+\rho v)^\frac{r_1-r}{r_1-\rho}\psi(z+r_1v)^\frac{r-\rho}{r_1-\rho}, \quad \rho<r<r_1.
\end{equation}
We now send $\rho\to -\infty$, to obtain a formula that holds for all $r\leq r_1$. It is clear to see that 
\begin{equation}
\label{eq:B_def}
\lim_{\rho\to-\infty}\psi(z+r_1 v)^\frac{r-\rho}{r_1-\rho}=\psi(z+r_1v)=:w(z).
\end{equation}

Recalling that $\psi$ is upper semi-continuous and that $0<\int \psi < +\infty$ yields $\psi$ is bounded by an exponential, we see that the function $W(x):=-\log \psi(x)$ is a proper, lower semi-continuous, convex function. Therefore, we can use the recession function from Definition~\ref{def:recession} to obtain
\begin{align*}
    \lim_{\rho\to-\infty}\frac{\log\psi(z+\rho v)}{\rho-r_1} &= \lim_{\rho\to-\infty}\frac{W(z+\rho v)}{r_1-\rho} = \lim_{t\to+\infty}\frac{W(z+r_1 v-t v)}{t}
    \\
    &= \lim_{t\to+\infty}\frac{W(z+r_1 v-t v)-W(z+r_1v)}{t}+\lim_{t \to +\infty}\frac{W(z+r_1v)}{t} 
    \\
    &=\lim_{t\to+\infty}\frac{W(z+r_1 v-t v)-W(z+r_1v)}{t} = W^+(-v),
\end{align*}
which, since $v$ is fixed, is a constant we call $a$. From this, we deduce that
\begin{equation}
\label{eq:A_def}
    \lim_{\rho\to-\infty}\psi(z+\rho v)^\frac{r_1-r}{r_1-\rho}=e^{a(r-r_1)}.
\end{equation}
Putting the pieces together, \eqref{eq:B_def}, \eqref{eq:A_def}, and \eqref{eq:log_affine} yield the claimed formula \eqref{eq:psi_affine_s_0_intro}.

We take a moment to show the converse direction for the equality case. Let $K=L+(-\pos(v))$, where $L\subset \{x\in\R^n : \langle x,u\rangle = r_1\}$ is a compact, convex set and $v\in\R^n$ is so that $\langle v,u\rangle=1$. Let $\mu$ be a probability measure on $\R^n$ with density $\psi$ satisfying, if  $x=z+rv$, where $z\in u^\perp$ and $r\leq r_1$, then $\psi(z+rv)=e^{c(r-r_1)}w(z)$ for some integrable, log-concave function $w$ on $L_u:=L-r_1v\subset u^\perp$. In fact, let $\nu$ be the Borel measure on $L_u$ whose density is $w$. Then, by \eqref{eq:change_variables_non_ortho} applied to $\mu$, we have
\begin{align*}
    \mu(K_r) = \int_{-\infty}^{r}\int_{L-t v}w(z)e^{c(t-r_1)}dzdt
    =\frac{1}{c}\int_{L_u}w(z)dz\int_{-\infty}^{c(r-r_1)}e^ada
    =\frac{1}{c}\nu(L_u)e^{c(r-r_1)}.
\end{align*}

By evaluating at $r=r_1$, we have $c= \nu(L_u)$, and so
\begin{equation}
\label{eq:log_mu_grun_equality_r}
\mu(K_r)=e^{\nu(L_u)(r-r_1)}
\end{equation}
Next, we compute $\langle g_{\mu},u\rangle$: again by \eqref{eq:change_variables_non_ortho}, but applied to the signed measure $\sigma$ with density $\langle x,u \rangle\psi(x)$, we have
\begin{equation}
\label{eq:log_centroid}
\begin{split}
    \langle g_{\mu},u\rangle \!=\!\int_{-\infty}^{r_1}\!t e^{\nu(L_u)(t-r_1)}dt\int_{L_u}w(z)dz
    \!=\!\frac{1}{\nu(L_u)}e^{-\nu(L_u)r_1}\int_{-\infty}^{\nu(L_u)r_1}\!\!\!\!\!\!\!\!t e^{t}dt  =r_1-\frac{1}{\nu(L_u)}.
\end{split}
\end{equation}
Therefore, by combining \eqref{eq:log_centroid} and \eqref{eq:log_mu_grun_equality_r}
\begin{align*}
\mu(H^{-})=\mu(K_{\langle g_{\mu},u\rangle})=e^{\nu(L_u)\left(\langle g_{\mu},u\rangle-r_1\right)}=\frac{1}{e},
\end{align*}
which is the sought-out equality.

\subsection{Equality characterization: the case when $-1 < s< 0$}

We now consider the final case, when $s\in (-1,0)$. First, we show that $m_\rho >1$. Again from the fact that $\mu(K_r)$ is $s$-affine, we have, from the inequality for arithmetic and geometric means,
\begin{align*}
\mu\left(K_{\frac{1}{2}\left((m_\rho+1)\rho-m_\rho r_1\right)+\frac{1}{2}r_1}\right) &= \left(\frac{1}{2} \mu\left(K_{\frac{1}{2}\left((m_\rho+1)\rho-m_\rho r_1\right)}\right)^s+\frac{1}{2}\mu(K)^s\right)^\frac{1}{s} 
\\
&\leq \left(\mu\left(K_{\frac{1}{2}\left((m_\rho+1)\rho-m_\rho r_1\right)}\right)\mu(K)\right)^\frac{1}{2}=\mu(K_{\rho}),
\end{align*}
with equality if and only if $\rho=r_1$. In the last line, we used \eqref{eq:log_equal}. Therefore, supposing $\rho <r_1$ we must have 
$\rho > \frac{1}{2}\left((m_\rho+1)\rho-m_\rho r_1\right)+\frac{1}{2}r_1$, which yields $m_\rho > 1$. In terms of $\zeta(s)$ from \eqref{eq:m_is_linear}, this is $\zeta(s)>0$. With this knowledge, we differentiate both sides of \eqref{eq:ms} in $\lambda$ and obtain
\begin{equation}
\label{eq:ms_s<0}
\frac{d}{d\lambda }m_{(1-\lambda)\rho+\lambda r_1} =1-m_\rho <0.
\end{equation}
Therefore, the map $\rho \mapsto m_\rho$ is decreasing for $\rho<r_1$ (to $1$ at $r_1$). Since this is the case, we continue the formula \eqref{eq:m_is_linear} for $m_\rho$ past $r_1$ until it reaches zero. Solving, one obtains $m_{\mathcal{R}}=0\longleftrightarrow\mathcal{R}=r_1+\frac{1}{\zeta(s)}$. We define a vector $b \in \operatorname{span}(v)$ such that $\langle b,u \rangle=\mathcal{R}$, i.e. $b:=\left(r_1+\frac{1}{\zeta(s)}\right)v$. Furthermore, we can replace $\zeta(s)$ with $(\mathcal{R}-r_1)^{-1}$ in the formula for $m_r$ from \eqref{eq:m_is_linear} and obtain
\begin{equation}
    \label{eq:m_is_linear_updated}
    m_r=\frac{\mathcal{R}-r}{\mathcal{R}-r_1}.
\end{equation}

Therefore, we obtain from \eqref{eq:updated_set_relations_3} the formula
\begin{equation}
\label{eq:sets_equal_s<0}
K_\rho=K+\frac{r_1-\rho}{\mathcal{R}-r_1}\left(K-b\right).
\end{equation}
Like in the previous instances, we take intersections of both sides of \eqref{eq:sets_equal_s<0} with $\{x:\langle x,u \rangle=\rho\}$ and obtain
\begin{equation}
\label{eq:psuedo_cone_final}
\begin{split}
\{x\in K : \langle x,u \rangle=\rho \}&= \left(K+\frac{r_1-\rho}{\mathcal{R}-r_1}\left(K-b\right)\right)\cap\ \{\langle x,u \rangle= \rho\}
\\
&= \left\{\left(1+\frac{r_1-\rho}{\mathcal{R}-r_1}\right)y-\frac{r_1-\rho}{\mathcal{R}-r_1}b:\langle y,u \rangle=r_1\right\} 
\\
&= \left(1+\frac{r_1-\rho}{\mathcal{R}-r_1}\right)L-\frac{r_1-\rho}{\mathcal{R}-r_1}b
=L-\frac{r_1-\rho}{\mathcal{R}-r_1}(b-L).
\end{split}
\end{equation}
Taking unions over \eqref{eq:psuedo_cone_final} for $\rho\leq r_1$, we obtain
\begin{align*}
    K&=L-\left\{(r_1-\rho)\frac{b-x}{\mathcal{R}-r_1}:x\in L, \rho\leq r_1\right\} 
    \\
    &=L-\left\{\lambda\frac{b-x}{\mathcal{R}-r_1}:x\in L, \lambda \geq 0\right\} 
    = L-\pos\left(\frac{b-L}{\mathcal{R}-r_1}\right)
\end{align*}

Our attention turns to deriving the density of $\mu$. We first re-write \eqref{eq:third_psi_affine}, inserting our formula \eqref{eq:m_is_linear_updated} for $m_r$: for every $z\in u^\perp$ and $\rho<r\leq r_1$, we have
\begin{equation}
\label{eq:psi_affine_neg_0}
    \psi(z+rv)=M_{\fcon}^{( \frac{r-\rho}{r_1-\rho})}\left(\psi\left(\left(\frac{\mathcal{R}-\rho}{\mathcal{R}-r}\right)z+\rho v\right), \psi\left(\left(\frac{\mathcal{R}-r_1}{\mathcal{R}-r}\right)z+r_1v\right)\right).
\end{equation}
Next, we use the fact that $W:=\psi^p$ is a proper, lower semi-continuous, convex function to obtain from \eqref{eq:psi_affine_neg_0}
\begin{equation}
\label{eq:psi_affine_neg}
    \psi(z+rv)\!=\!\left(\left( \frac{r_1-r}{r_1-\rho}\right)W\left(\left(\frac{\mathcal{R}-\rho}{\mathcal{R}-r}\right)z+\rho v\right) \!+\! \left( \frac{r-\rho}{r_1-\rho}\right) W\left(\left(\frac{\mathcal{R}-r_1}{\mathcal{R}-r}\right)z+r_1v\right)\right)^\frac{1}{p}
\end{equation}
for every $z\in u^\perp$ and $\rho< r\leq r_1$. We must now carefully send $\rho \to -\infty$ in \eqref{eq:psi_affine_neg}. First, we obtain from \eqref{eq:recession} and the statement afterwards that, by performing the variable substitution $\rho=r_1-t$,

\begin{align*}
    & \lim_{\rho\to-\infty}\frac{W\left(\left(\frac{\mathcal{R}-\rho}{\mathcal{R}-r}\right)z+\rho v\right)}{r_1-\rho}=\lim_{t\to \infty}\frac{W\left(\left(\frac{\mathcal{R}-r_1+t}{\mathcal{R}-r}\right)z+(r_1-t )v\right)}{t}
    \\
    &=\lim_{t\to \infty}\frac{W\left(\left(\frac{\mathcal{R}-r_1}{\mathcal{R}-r}\right)z+r_1v+t\left(\frac{1}{\mathcal{R}-r}z-v\right)\right)}{t}
    \\
   &=\lim_{t\to\infty}\left(\frac{W\left(\left(\frac{\mathcal{R}-r_1}{\mathcal{R}-r}\right)z+r_1v+t\left(\frac{1}{\mathcal{R}-r}z-v\right)\right)-W\left(\left(\frac{\mathcal{R}-r_1}{\mathcal{R}-r}\right)z+r_1v \right)}{t}\right)
   \\
   &=W^+\left(\frac{z}{\mathcal{R}-r}-v\right).
\end{align*}
Consequently, sending $\rho\to-\infty$ in \eqref{eq:psi_affine_neg} yields
\begin{equation}
\label{eq:psi_affine_neg_final_psych!}
\begin{split}
    \psi(z+rv)=\left((r_1-r)W^+\left(\frac{z}{\mathcal{R}-r}-v\right)+W\left(\frac{\mathcal{R}-r_1}{\mathcal{R}-r}z+r_1v\right)\right)_+^\frac{1}{\fcon}\chi_{(-\infty,r_1)}(r).
\end{split}
\end{equation}
We claim it suffices to show that there exists $a>0$ such that, for every $z\in u^\perp$ and $r\leq r_1$,
\begin{equation}
\label{eq:w_relate}
W\left(\frac{\mathcal{R}-r_1}{\mathcal{R}-r}z+r_1v\right) = \frac{1}{a}W^+\left(\frac{z}{\mathcal{R}-r}-v\right).
\end{equation}
Indeed, supposing this is the case, inserting \eqref{eq:w_relate} into \eqref{eq:psi_affine_neg_final_psych!} yields the formula
\begin{equation}
\label{eq:psi_affine_neg_final}
    \psi(z+rv)=\left(1+a(r-r_1)\right)_+^\frac{1}{\fcon}W\left(\frac{\mathcal{R}-r_1}{\mathcal{R}-r}z+r_1v\right)^\frac{1}{p}\chi_{(-\infty,r_1)}(r).
\end{equation}
Setting $w(y) = W\left(y+r_1v\right)^\frac{1}{p}$ for $y\in u^\perp$ then establishes the claimed formula \eqref{eq:psi_affine_neg_intro}.

To establish \eqref{eq:w_relate}, we will further explore \eqref{eq:almost_first_psi_affine}. We first re-write this formula, inserting \eqref{eq:v_def} to replace $b_\rho$ with $v$ and using \eqref{eq:Q_s} to replace $\mu(K_\rho)$. We obtain, for every $x \in K$, the identity, for every $\rho \leq r_1$,
\begin{equation}
\label{eq:new_psi}
    \psi(x)=\frac{m_\rho^n}{(1+cs(\rho-r_1))_+^\frac{1}{s}}\psi(m_\rho x+(\rho-m_\rho r_1)v)
\end{equation}
We now insert the formula for $m_\rho$ from \eqref{eq:m_is_linear_updated} and the relation $s=\frac{p}{1+np}$ to obtain, for every $x\in K$
\begin{equation}
\label{eq:new_psi_2}
    \psi(x)=\frac{\left(\frac{\mathcal{R}-\rho}{\mathcal{R}-r_1}\right)^n}{\left(1+c\left(\frac{p}{1+np}\right)(\rho-r_1)\right)_+^\frac{1+np}{p}}\psi\left(\left(\frac{\mathcal{R}-\rho}{\mathcal{R}-r_1}\right) x+\left(\frac{\rho-r_1}{\mathcal{R}-r_1}\right) \mathcal{R}v\right).
\end{equation}
In \eqref{eq:new_psi_2}, we interchange $\psi$ with $W$ via the relation $W=\psi^p$ to obtain
\begin{equation}
\label{eq:new_psi_3}
    W(x)=\frac{\left(\frac{\mathcal{R}-\rho}{\mathcal{R}-r_1}\right)^{np}}{\left(1+c\left(\frac{p}{1+np}\right)(\rho-r_1)\right)_+^{1+np}}W\left(\left(\frac{\mathcal{R}-\rho}{\mathcal{R}-r_1}\right) x+\left(\frac{\rho-r_1}{\mathcal{R}-r_1}\right) \mathcal{R}v\right).
\end{equation}
Next, we send $\rho\to -\infty$ in \eqref{eq:new_psi_3} to obtain a formula for $W(x)$ independent of $\rho$. By performing the change of variables $t=-\rho$, we obtain, for every $x\in K$, the formula
\begin{align*}
    W(x)&=\lim_{\rho\to-\infty}\frac{\left(\frac{\mathcal{R}-\rho}{\mathcal{R}-r_1}\right)^{np}}{\left(1+c\left(\frac{p}{1+np}\right)(\rho-r_1)\right)^{1+np}}W\left(\left(\frac{\mathcal{R}-\rho}{\mathcal{R}-r_1}\right) x+\left(\frac{\rho-r_1}{\mathcal{R}-r_1}\right) \mathcal{R}v\right) 
    \\
    &=\lim_{t\to\infty}\frac{\left(\mathcal{R}+t\right)^{np}t}{\left(1+c\left(\frac{|p|}{1+np}\right)(t+r_1)\right)^{1+np}(\mathcal{R}-r_1)^{np}}\frac{W\left(\left(\frac{\mathcal{R}}{\mathcal{R}-r_1}\right) (x-r_1v)+t\left(\frac{x-\mathcal{R}v}{\mathcal{R}-r_1}\right)\right)}{t}
    \\
    &=(\mathcal{R}-r_1)\left(\frac{c|p|(\mathcal{R}-r_1)}{1+np}\right)^{-(1+np)}W^+\left(\frac{x}{\mathcal{R}-r_1}-\frac{\mathcal{R}}{\mathcal{R}-r_1}v\right).
\end{align*}
By setting 
\begin{equation}
\label{eq:neg_a}
a=\frac{1}{\mathcal{R}-r_1}\left(\frac{c|p|(\mathcal{R}-r_1)}{1+np}\right)^{1+np},
\end{equation} 
we obtain, for every $x\in K$,
\begin{equation}
\label{eq:new_psi_4}
    W\left(x\right) = \frac{1}{a}W^+\left(\frac{x}{\mathcal{R}-r_1}-\frac{\mathcal{R}}{\mathcal{R}-r_1}v\right)
\end{equation}
setting $x=\frac{\mathcal{R}-r_1}{\mathcal{R}-r}z+r_1v$ in \eqref{eq:new_psi_4} yields \eqref{eq:w_relate}, completing the proof.

\qed

\subsection{The remaining regime of concavity}
\label{sec:very_neg}
This section is dedicated to proving Proposition~\ref{p:no_bound}. We recall the content of this proposition is that only the trivial bound (of $0$) holds for $s$-concave measures with $s \leq -1$. We must show that there exists a sequence $\{\mu_k\}_{k=1}^\infty$ of $s$-concave probability measures on $\R$ such that
    \[
    \lim_{k\to +\infty }\mu_k((-\infty,g_{\mu_k}])=0.
    \]
    Fix $p \in (-1,-\frac{1}{2})$. We start by defining the function $h_k$ given by, for $k\geq 2$,
    \[
h_k(t)=(1-t)^{\frac{1}{p}}\chi_{_{[0,1-\frac{1}{k}]}}(t),
    \]
    and then let $\mu_k$ be the $s=\left(\frac{p}{1+p}\right)$-concave probability measure on $\R$ with density $\frac{h_k}{\int_{\R} h_k(t)dt}$. Notice that $s < -1$ and
    \[
    \int_{\R} h_k(t)dt=\int_0^{1-\frac{1}{k}} (1-t)^\frac{1}{p}dt= \frac{p}{p+1}\left(1-\left(\frac{1}{k}\right)^\frac{p+1}{p}\right).
    \]
    We compute the barycenter and obtain
    \[
    g_{\mu_k} = \frac{\int_0^{1-\frac{1}{k}}t(1-t)^\frac{1}{p}dt}{\int_0^{1-\frac{1}{k}}(1-t)^\frac{1}{p}dt}=1-\left(\frac{p+1}{2p+1}\right)\left(\frac{1-\left(\frac{1}{k}\right)^\frac{2p+1}{p}}{1-\left(\frac{1}{k}\right)^\frac{p+1}{p}}\right).
    \]
    Computing $\mu_k((-\infty,g_{\mu_k}])$, we obtain
    \begin{align*}
    \mu_k((-\infty,g_{\mu_k}]) &= \frac{\int_0^{g_{\mu_k}}(1-t)^\frac{1}{p}dt}{\int_0^{1-\frac{1}{k}}(1-t)^\frac{1}{p}dt}=\frac{1-\left(\left(\frac{p+1}{2p+1}\right)\left(\frac{1-\left(\frac{1}{k}\right)^\frac{2p+1}{p}}{1-\left(\frac{1}{k}\right)^\frac{p+1}{p}}\right)\right)^\frac{p+1}{p}}{\left(1-\left(\frac{1}{k}\right)^\frac{p+1}{p}\right)}
    \\
    &=\frac{\left(\frac{1}{k}\right)^{-\frac{p+1}{p}}-\left(\left(\frac{p+1}{2p+1}\right)\left(\frac{1-\left(\frac{1}{k}\right)^\frac{2p+1}{p}}{\left(\frac{1}{k}\right)-\left(\frac{1}{k}\right)^\frac{2p+1}{p}}\right)\right)^\frac{p+1}{p}}{\left(\left(\frac{1}{k}\right)^{-\frac{p+1}{p}}-1\right)}.
    \end{align*}
    Since $p\in (-1,-\frac{1}{2})$, the terms $\left(\frac{1}{k}\right)^{-\frac{p+1}{p}}$ and $\left(\frac{1}{k}\right)^{\frac{2p+1}{p}}$ will go to zero as $k\to \infty$, therefore,
    \[
    \lim_{k\to +\infty}\mu((-\infty,g_{\mu_k}]) = \lim_{k\to+\infty}\left(k\left|\frac{p+1}{2p+1}\right|\right)^\frac{p+1}{p}=0.
    \]

    Similarly, if $s=-1$, then $p=-\frac{1}{2}$.
    Repeating the above steps, we get
    \[
    \int_{\R} h_k(t)dt = k-1
   \quad \text{and} \quad
    g_{\mu_k}= 1-\frac{\log(k)}{k-1}.
    \]
    Therefore, one can see that
    \[
    \mu_k((-\infty,g_{\mu_k}]) = \frac{1}{\log (k)}-\frac{1}{k-1}.
    \]
    Taking the limit, we again have
    \[
    \lim_{k\to +\infty}  \mu_k((-\infty,g_{\mu_k}]) = 0.
    \]
\qed

\section{The Gaussian measure}
\label{sec:gaussian}
We start by proving our main theorem concerning the standard Gaussian measure.

\begin{proof}[Proof of Theorem~\ref{t:Gaussian_Grun}]
    The inequality \eqref{eq:Grunbaum_Gaussian} follows from Theorem~\ref{t:F_concave} by using using that $\gamma_n$ is $\Phi^{-1}$-concave and the fact that $\Phi^{-1}$ is increasing with primitive $-I_{\gamma}$.  For the equality case, we must have that
    \begin{align*}
    \Phi^{-1}\circ \gamma_n \left(\left\{x\in K:\left\langle x,u\right\rangle \leq r\right\}\right)\chi_{[-h_K(-u),h_K(u)]}(r)
    \end{align*}
    is affine. From the equality case of Ehrhard's inequality, the collection of sets $$\{x\in K :\langle x,u \rangle\leq r\}_{r\in [-h_K(-u),h_K(u)]}$$ are nested slabs containing $H^{-1}$. Thus, $K=\{x:\langle x,u\rangle \leq a\}$ for some $a\in\R$.
\end{proof}

The following is an application of Theorem~\ref{t:Gaussian_Grun}, which takes into account that $\Phi^{-1}$ can be negative.

\begin{cor}
    Let $K$ be a convex set in $\R^n$ with non-empty interior. Suppose $H$ is a hyperplane through the Gaussian barycenter of $K$. Then,
    \begin{equation}
        |\Phi^{-1}(\gamma_{n}(K\cap H^{-}))| \leq \frac{I_{\gamma}(t)}{t}; \quad \text{with}\; t=\gamma_n(K).
        \label{eq:Grunbaum_Gaussian_2}
    \end{equation}
\end{cor}
\begin{proof}
    For ease set $x=\gamma_n(K\cap H^-)$. If $x\leq \frac{1}{2}$, then, we use \eqref{eq:Grunbaum_Gaussian} directly. By observing that $\Phi^{-1}(x)=-|\Phi^{-1}(x)|$, the inequality \eqref{eq:Grunbaum_Gaussian_2} follows. 
    
    It is clear that \eqref{eq:Grunbaum_Gaussian} holds with $H^+$ instead of $H^-$. Therefore, if $x\geq \frac{1}{2}$, we use \eqref{eq:Grunbaum_Gaussian} in this way and the inequality
    \begin{align*}
        \Phi^{-1}(\gamma_n(K\cap H^+)) = \Phi^{-1}(\gamma_n(K)-x)\leq \Phi^{-1}(1-x) =-\Phi^{-1}(x),
    \end{align*}
    to arrive at \eqref{eq:Grunbaum_Gaussian_2}.
\end{proof}

We now show that the $n=1$ case of Theorem~\ref{t:Gaussian_Grun} has implications for other measures on $\R$. We first return to the discussion of isoperimetry. One may ask if the following inequality holds
\begin{equation}
\label{eq:non_local_mu_iso}
\mu(A+hB_2^n)\geq \Phi\left(\Phi^{-1}(\mu(A))+h\right), \quad h>0,
\end{equation}
where $\mu$ is a Borel measure on $\R^n$. We emphasize here that $\Phi$ is the usual CDF of the Gaussian measure. Bobkov \cite{SB02} showed that, if $\mu$ is log-concave with respect to the Gaussian, then \eqref{eq:non_local_mu_iso} holds for all Borel $A$ and $h>0$. 

Motivated by this result, we want to establish a version of Theorem~\ref{t:Gaussian_Grun} for other measures besides the Gaussian measure in dimension $1$. We say $\mu$ is \textit{$\gamma$-transport} concave, if $\Phi^{-1}\circ \Phi_{\mu}$ is concave. Let us explain where this notion comes from. Let $T$ transport $\gamma$ to $\mu$, that is
    \begin{equation}
        \int_{-\infty}^{T(s)}\psi(u)du=\int_{-\infty}^s\phi(u)du,
        \label{eq:optimal}
    \end{equation}
where $\phi: \R \to \R_+$ is given by $\phi(s) = \frac{1}{\sqrt{2\pi}}e^{-\frac{s^2}{2}}$. This can be written as $\Phi_{\mu}(T(s))=\Phi(s)$. Solving for $T$, we obtain $T=\Phi_{\mu}^{-1}\circ \Phi$, or equivalently $T^{-1}=\Phi^{-1}\circ\Phi_{\mu}$. Therefore, the statement that $\mu$ is \textit{$\gamma$-transport} concave is precisely that $T^{-1}$ is concave, i.e. $T$ is convex.

We recall for later that optimal transport satisfies a Monge-Amp\`ere equation: by performing a change of variables, we obtain from \eqref{eq:optimal}
    \begin{equation}
    \label{eq:optimal_change}
        \int_{-\infty}^{s}\psi\left(T(u)\right)T^\prime(u)du=\int_{-\infty}^s\phi(u)du.
    \end{equation}
    Differentiating both sides of \eqref{eq:optimal_change} in $s$, one obtains for almost all $s\in \R$:
    \begin{equation}
        \psi\left(T(s)\right)T^\prime(s)=\phi(s).
        \label{eq:monge}
    \end{equation}
    From \eqref{eq:monge}, we obtain the usual definition of optimal transport: for every $a\leq b$, one has
    \begin{align*}
        \mu((a,b))&=\int_{a}^b\psi(r)dr = \int_{T^{-1}(a)}^{T^{-1}(b)}\psi(T(s))T^\prime(r)dr
         = \int_{T^{-1}(a)}^{T^{-1}(b)}\varphi(r)dr
        \\
        & = \gamma((T^{-1}(a),T^{-1}(b)))=\gamma(T^{-1}((a,b))).
    \end{align*}

There are several examples of $\gamma$-transport concave measures. Indeed, let $T: \R \to \R\cup\{+\infty\}$ be a convex, increasing, invertible function. Then, take $\psi : \R \to \R_+$ to be the function given by 
\begin{equation}
\label{transport_concave_formula}
\psi(t)=\frac{\varphi(T^{-1}(t))}{T^\prime(T^{-1}(t))}
\end{equation}
and $\mu$ be the Borel measure on $\R$ with density $\psi$. We have that the CDF of $\mu$, denoted as $\Phi_{\mu}$, is given by $\Phi_{\mu}(s)=\Phi(T^{-1}(s))$, $T$ transports $\gamma$ onto $\mu$ (by construction, $T$ will satisfy \eqref{eq:monge}), and $T$ is convex. 

We list an explicit example; consider the convex function $T(t)=te^t$. The inverse of this function is supported on $[-\frac{1}{e},+\infty)$ and is often denoted $W$, the so-called Lambert function. That is, the Lambert function satisfies $t=W(t)e^{W(t)}.$ Then, the measure $\mu$ with density $\psi(0)=\frac{1}{\sqrt{2\pi}}$ and, for $t\in [-\frac{1}{e},+\infty)\setminus\{0\},$ $$\psi(t)=\frac{\varphi(W(t))}{t+e^{W(t)}}=\frac{1}{\sqrt{2\pi t^2}}\frac{W(t)}{W(t)+1}e^{-\frac{W(t)^2}{2}}$$ is $\gamma$-transport concave.

% Another example, taking the convex function $T(t)=t+e^t$, its inverse $W$ satisfies $t=W(t)+e^{W(t)}$, and the measure $\mu$ with density 
% \[
% \psi(t)=\frac{\varphi(W(t))}{1+t-W(t)}
% \]
% is $\gamma$-transport concave. 

Next, we show that Theorem~\ref{t:Gaussian_Grun} holds for every measure that is $\gamma$-transport concave.
\begin{cor}
    Let $\mu$ be a probability measure on $\R$ with density $\psi$. Suppose $\mu$ is $\gamma$-transport concave. Consider an interval $(a,b) \subset \R$ such that $\mu((a,b))>0$ and set
    $$g = \frac{\int_a^br\psi(r)dr}{\int_a^b\psi(r)dr} \in (a,b) \quad \text{and} \quad t=\int_a^b\psi(r)dr.$$
    Then,
    \begin{equation}
        \mu((a,g])\geq \Phi\left(-\frac{I_\gamma(t)}{t}\right).
        \label{eq:intersect}
    \end{equation}
    \label{cor:intersect}
\end{cor}
\begin{proof}
We will use the fact that $T= \Phi_{\mu}^{-1} \circ \Phi$ is the optimal transport map from $\gamma$ onto $\mu$. Let $B=(a,b)$ and observe that
\begin{equation}
\begin{split}
    g_{\gamma}(T^{-1}B)&:=\frac{\int_{T^{-1}(a)}^{T^{-1}(b)}r\phi(r)dr}{\int_{T^{-1}(a)}^{T^{-1}(b)}\phi(r)dr}
    =\frac{\int_{a}^{b}T^{-1}(r)\psi(r)dr}{\int_{a}^{b}\psi(r)dr}
    \\
    &\leq T^{-1}\left(\frac{\int_{a}^{b}r\psi(r)dr}{\int_{a}^{b}\psi(r)dr}\right)=T^{-1}(g_{\mu}(B)),
\end{split}
\end{equation}
where we used \eqref{eq:monge} and Jensen's inequality. Recall the fact that $\mu$ being $\gamma$-transport concave is equivalent to $T^{-1}$ being concave.
Applying $T$ to both sides, and using that $T$ is increasing, we obtain
\[
Tg_{\gamma}(T^{-1}B) \leq g_{\mu}(B),
\]
and deduce
\[
\mu((-\infty,g_\mu(B)])\geq \mu((-\infty,Tg_{\gamma}(T^{-1}B)]) = \gamma((-\infty,g_{\gamma}(T^{-1}B)]).
\]
We then have from Theorem~\ref{t:Gaussian_Grun}, with $t=\gamma(T^{-1}(B))=\mu(B),$
 \begin{equation}
 \begin{split}
        \gamma_{1}((-\infty,g_\gamma(T^{-1}B)]) \geq \Phi\left(-\frac{I_{\gamma}(t)}{t}\right).
\end{split}
    \end{equation}
which yields the claim.     
\end{proof}
\begin{rem}
    We saw in Lemma \ref{l:half_space_concave} that one-dimensional convex measures satisfy the Ehrhard-type inequality \eqref{eq:half_space_convex} with $\Phi_{\mu}$. However, if  a Borel measure $\mu$ on $\R$ is both $\gamma$-transport concave and convex, then the concavity of $\Phi^{-1}\circ\Phi_\mu$ combined with \eqref{eq:half_space_convex} yields that $\mu$ satisfies the Gaussian-Ehrhard-type inequality:
    \begin{equation}
    \begin{split}
        &\mu\left((-\infty, (1-\lambda)a+\lambda b)\right)
        \\
        &\geq \Phi\left((1-\lambda)\Phi^{-1}(\mu((-\infty,a)))+\lambda \Phi^{-1}(\mu((-\infty,b)))\right).
        \label{eq:half_space_convex-Gaussian}
        \end{split}
    \end{equation}

    By applying directly Theorem~\ref{t:F_concave}, we then arrive directly at \eqref{eq:intersect}. However, we see that Theorem~\ref{t:one_d} is stronger. Indeed, if $\mu$ is $\gamma$-transport concave, then $\Phi^{-1} \circ \Phi_{\mu}$ is concave; then, if $\mu$ is convex, we apply Jensen's inequality to the conclusion of Theorem~\ref{t:one_d} to deduce \eqref{eq:intersect}:
    $$\Phi^{-1}(\mu((a,g))) \geq (\Phi^{-1}\circ \Phi_\mu)\left(\int_0^t\Phi_{\mu}^{-1}(r)\frac{dr}{t}\right) \geq \int_0^t\Phi^{-1}(r)\frac{dr}{t}.$$
\end{rem}

Recall that a measure $\mu$ is convex if and only if the reciprocal of its density is convex. Letting $\psi$ be defined by \eqref{transport_concave_formula}, one can take two derivatives of its reciprocal and see that $\gamma$-transport concave measures are not necessarily convex measures.

We finish by showing that a measure $\mu$ that is even and $\gamma$-transport concave must be centered Gaussian with arbitrary variance. Although this does not necessarily discredit the possibility of Corollary~\ref{cor:intersect} holding for measures that are even, as being $\gamma$-transport concave is sufficient but not necessary; we do not know how to relate the two barycenters without this condition.
\begin{prop}
    Let $\mu$ be an even probability measure on $\R$ with density. Then, it is $\gamma$-transport concave if and only if there exists $\sigma >0$ such that $\mu$ has density
    \[
    \frac{d\mu(r)}{dr}=\frac{1}{\sqrt{2\pi \sigma^2}}e^{-\frac{r^2}{2\sigma^2}}.
    \]
\end{prop}
\begin{proof}
Let $T$ be the optimal transport map from $\gamma$ onto $\mu$, i.e. $T$ satisfies $\Phi_{\mu}(T(r))=\Phi(r)$ and goes from $\R$ to $\supp(\mu)$. First, we have that $T$ must be odd: we see that
\[
\Phi_{\mu}(T(-r))=\Phi(-r)=1-\Phi(r) \longrightarrow 1-\Phi_{\mu}(-T(-r))=1-\Phi(r).
\]
Thus, we have $\Phi_{\mu}(-T(-r))=\Phi(r)$, which yields $T(-r)=-T(r).$ Also, we see that $0$ is a fixed point of $T$. By hypothesis, $T$ is convex, which means $r\mapsto \frac{T(r)-T(0)}{r}=\frac{T(r)}{r}$ is increasing on $\supp(\mu)$, which is a symmetric interval. But, $r\mapsto \frac{T(r)}{r}$ is even (since $T(r)$ and the identity map are odd); an even, increasing function must be constant. Since $T$ fixes $0$, we must have $T(r)=\sigma r$. Since $T$ is increasing, $\sigma >0$.

Next, it is easy to see we must have $\supp(\rho)=\R$. Indeed, write $\supp(\rho)=(-\alpha_\mu,\alpha_\mu)$. Then, the formula $\Phi_{\mu}(r)=\Phi(T^{-1}(r))$ yielding $$1=\lim_{r\to (\alpha_\mu)^{-}}\Phi_{\mu}(r)=\lim_{r\to (\alpha_\mu)^{-}}\Phi(T^{-1}(r))$$ 
tells us that $\lim_{r\to (\alpha_\mu)^{-}} T^{-1}(r)=+\infty$. But, $T^{-1}(r)=\frac{1}{\sigma}{r}$ on $\supp(\rho)$, and so $\alpha_\mu=+\infty$.

Now that we know $\supp(\rho)=\R$, we are ready to construct $\mu$. Let $\psi$ be the density of $\mu.$ Then: $T^\prime(r)=\frac{\varphi(r)}{\psi(T(r))}$. Inserting our formula for $T$ yields
\[
\psi(r)= \frac{1}{\sigma}\varphi(\sigma^{-1}r)=\frac{1}{\sqrt{2\pi \sigma^2}}e^{-\frac{r^2}{2\sigma^2}},
\]
as claimed.
\end{proof}

{\bf Funding:} 
D. Langharst was funded by the Fondation Sciences math\'ematiques de Paris post-doctoral program.

J. Liu was funded by the China National Postdoctoral Program for Innovative Talents of CPSF (BX20240102), and by the National Natural Science Foundation of China (12401252).

J. Liu and S. Tang thank the LAMA for its financial support.

F. Mar\'in Sola was funded by the grant PID2021-124157NB-I00, funded by MCIN/AEI/ 10.13039/ 501100011033/``ERDF A way of making Europe'', as well as by the grant  ``Proyecto financiado por la CARM a trav\'es de la convocatoria de Ayudas a proyectos para el desarrollo de investigaci\'on cient\'ifica y t\'ecnica por grupos competitivos, incluida en el Programa Regional de Fomento de la Investigaci\'on Cient\'ifica y T\'ecnica (Plan de Actuaci\'on 2022) de la Fundaci\'on S\'eneca-Agencia de Ciencia y Tecnolog\'ia de la Regi\'on de Murcia, REF. 21899/PI/22''.

{\bf Acknowledgments:} This project was completed while the third, fourth and fifth named authors were visiting the first named author at Universit\'e Gustave Eiffel. We would like to whole-heatedly thank the university and the staff for their hospitality and kindness. 

\printbibliography

\noindent Matthieu Fradelizi
\\
Univ Gustave Eiffel, Univ Paris Est Creteil, CNRS LAMA UMR8050, F-77447 Marne-la-Vall\'ee, France.
\\
E-mail address: matthieu.fradelizi@univ-eiffel.fr
\vspace{2mm}
\\
\noindent Dylan Langharst
\\
Institut de Math\'ematiques de Jussieu, Sorbonne Universit\'e, 4 Place Jussieu, 75005 Paris, France.
\\
E-mail address: dylan.langharst@imj-prg.fr
\vspace{2mm}
\\
\noindent Jiaqian Liu 
\\
School of Mathematics and statistics, Henan University, 475001, Kaifeng, China
\\
E-mail address: liujiaqian@henu.edu.cn
\vspace{2mm}
\\
\noindent Francisco Mar\'in Sola
\\
Departamento de Ciencias, Centro Universitario de la Defensa (CUD), 30729 San Javier (Murcia), Spain. 
\\ Email address: francisco.marin7@um.es
\vspace{2mm}
\\
\noindent Shengyu Tang
\\
Institute of Mathematics, Hunan University, 410082, Changsha, China
\\
E-mail address: tsy@hnu.edu.cn
\vspace{2mm}
\end{document}